\documentclass{amsart} 

\usepackage[utf8]{inputenc}
\usepackage{amsthm}
\usepackage{amssymb}
\usepackage{amsmath}
\usepackage{stmaryrd}
\usepackage{cite}
\usepackage{hyperref}
\usepackage{MnSymbol} 
\usepackage{enumerate} 
\usepackage{array} 
\usepackage[enableskew]{youngtab} 
\usepackage{appendix}
 \usepackage{graphicx}
 \usepackage{multirow}
 
\usepackage[all]{xy}

\theoremstyle{plain}
\newtheorem{thm}{Theorem}

\newtheorem{prop}{Proposition}
\newtheorem{lem}[prop]{Lemma}
\newtheorem{cor}[prop]{Corollary}

\theoremstyle{definition}
\newtheorem{df}[prop]{Definition}
\newtheorem{claim}{Claim}

\newtheorem*{ex}{Example}
\newtheorem*{assumptions}{Assumptions}

\theoremstyle{remark}
\newtheorem*{rk}{Remark}

\newcommand{\arcd}{\ar@{-}@/_/} 
\newcommand{\arcu}{\ar@{-}@/^/} 
\newcommand{\tra}{\ar@{-}} 

\newcommand{\xysmall}{\xymatrixrowsep{5pt}\xymatrixcolsep{10pt}\xymatrix} 
\newcommand{\xymini}{\xymatrixrowsep{5pt}\xymatrixcolsep{7pt}\xymatrix}

\newcommand{\xytiny}{\xymatrixrowsep{1pt}\xymatrixcolsep{3pt}\xymatrix}

\renewcommand{\top}{\text{top}}
\newcommand{\bottom}{\text{bottom}}

\newcommand{\tensorel}{\underset{\scriptscriptstyle e_lAe_l}{\otimes}}
\newcommand{\tensorover}[1]{\underset{\scriptscriptstyle #1}{\otimes}}

\renewcommand{\mod}{\text{\textnormal{mod}}} 
\newcommand{\Hom}{\text{\textnormal{Hom}}}
\newcommand{\Ext}{\text{\textnormal{Ext}}} 
\newcommand{\End}{\text{\textnormal{End}}}

\renewcommand{\th}{\text{\textnormal{\small th}}\hspace*{.6ex}}

\usepackage{amsaddr}

\begin{document}

\title{Permutation modules for cellularly stratified algebras}
\author{Inga Paul}
\address{Institut f\"ur Algebra und Zahlentheorie, Universit\"at Stuttgart}
\email{inga.paul@mathematik.uni-stuttgart.de}

\date{\today}

\begin{abstract}
Permutation modules play an important role in the representation theory of the symmetric group. Hartmann and Paget defined permutation modules for non-degenerate Brauer algebras. We generalise their construction to a wider class of algebras, namely cellularly stratified algebras, satisfying certain conditions. Partition algebras are shown to satisfy these conditions, provided the characteristic of the underlying field is large enough. Thus we obtain a definition of permutation modules for partition algebras. 
\end{abstract}

\maketitle

\textbf{Keywords.} cellular algebras, permutation modules, Young modules, partition algebras\\
 
\section{Introduction}
The Specht modules $S^\lambda$ are cornerstones of the representation theory of symmetric groups $\Sigma_r$. In characteristic zero, they form a complete set of simple modules (\hspace*{-3pt}\cite[Theorem 3]{Jirred}). In arbitrary characteristic $p$, the simple modules occur as top quotients $S^\lambda / S^\lambda \cap {S^\lambda}^\perp$ of Specht modules, in case $\lambda$ is a $p$-regular partition\footnote{For $p$-singular partitions $\lambda$, $S^\lambda / S^\lambda \cap {S^\lambda}^\perp$ is zero.} of $r$ (\hspace*{-3pt}\cite[Theorem 2]{Jirred}). In the more general case of cellular algebras, introduced by Graham and Lehrer \cite{GL} in 1996, the cell modules $\Theta(\lambda)$ take the role of Specht modules $S^\lambda$ or their duals $S_\lambda$. 

Another cornerstone in the representation theory of symmetric groups are the permutation modules $M^\lambda = k\Sigma_r \tensorover{k\Sigma_\lambda} k$. By James' Submodule Theorem (\hspace*{-3pt}\cite[Theorem 1]{Jirred}), $M^\lambda$ has a unique direct summand $Y^\lambda$, called Young module, containing $S^\lambda$ as a submodule. Since Young modules are self-dual (cf. \cite[2.2.1 (b)]{ErdSchurFinalType}), $Y^\lambda$ can also be characterised as the only direct summand of $M^\lambda$ with quotient $S_\lambda$. Young modules for different partitions are non-isomorphic (\hspace*{-3pt}\cite[Theorem 3.1 (iii)]{Jtriv}). All direct summands of $M^\lambda$ are Young modules $Y^\mu$, with $\mu \leq \lambda$ and $Y^\lambda$ appears exactly once (\hspace*{-3pt}\cite[Theorem 3.1 (i)]{Jtriv}). \\

Cellularly stratified algebras, introduced by Hartmann, Henke, K\"onig and Paget (\hspace*{-3pt}\cite{HHKP}) in 2010, are cellular algebras with additional structure. The aim of this article is to generalise the well-known results about permutation modules for symmetric groups to cellularly stratified algebras containing group algebras of symmetric groups, or their Hecke algebras, as subalgebras. Young modules for cellularly stratified algebras have already been used in \cite{HHKP}. They were defined abstractly via iterated universal extensions.  While this definition is useful for theoretical considerations, the construction of iterated universal extensions might be hard in examples. Extending the construction of Young modules for Brauer algebras of Hartmann and Paget (\hspace*{-3pt}\cite{HP}), we present an explicit construction of Young modules (Theorem \ref{def: Y}), which coincides, under additional assumptions stated in Section \ref{assumptions}, with the abstract definition in \cite{HHKP} (Corollary \ref{cor: young modules coincide}). This provides new proofs for some results of \cite{HHKP}, e.g. a method of finding all indecomposable (relative) projective modules (\hspace*{-3pt}\cite[Proposition 12.3]{HHKP}) and Schur-Weyl duality (\hspace*{-3pt}\cite[Theorem 13.1]{HHKP}). The fact that two Young modules with different indices are non-isomorphic follows from the construction (Corollary \ref{cor: Young modules not isomorphic}). \\

The structural main result of this article is the decomposition of permutation modules $M(l,\lambda)$ into Young modules $Y(m,\mu)$ (Theorem \ref{main thm}). In order to decompose permutation modules for symmetric groups, James used Schur algebras via Schur-Weyl duality and PIMs. There is a Schur-Weyl duality between cellularly stratified algebras and certain quasi-hereditary algebras, which can be regarded as Schur algebras associated to the cellularly stratified algebras, by \cite[Theorem 13.1]{HHKP}.  

Our homological main result is to show that the Young modules $Y(l,\lambda)$ admit filtrations by cell modules (Theorem \ref{thm: cell filtration}) and are relative projective in the category $\mathcal{F}(\Theta)$ of modules admitting cell filtrations (Theorem \ref{thm: relative projective}). These statements hold provided the cellularly stratified algebra satisfies the additional assumptions stated in Section \ref{assumptions}. This generalises a result from Hemmer and Nakano \cite[Proposition 4.1.1]{HN} for Hecke algebras and enables us to prove the analogue of James' theorem on the decomposition of permutation modules.  \\

This article was inspired by the results of Hartmann and Paget \cite{HP} for Brauer algebras. We apply the theory developed here to Brauer algebras (Section \ref{sec: applications Brauer algebras}) and recover their results (Theorem \ref{thm: Brauer algebras}), thus providing new proofs. 

Further applications to partition algebras (Section \ref{sec: applications partition algebras}) show that, provided the characteristic of the field is large enough, we can construct permutation modules for partition algebras with the desired properties (Theorem \ref{thm: partition algebras}). In order to have the homological Hemmer-Nakano-type results, we need filtrations of restrictions of cell modules to symmetric groups (\hspace*{-3pt}\cite[Theorem 1]{arxiv}) and filtrations of restrictions of permutation modules to symmetric groups. In Proposition \ref{prop: partition algebras} we show that the restriction of a permutation module to a group algebra of a symmetric group is isomorphic to a direct sum of permutation modules over this symmetric group. You can find an example (\ref{appendix: example}) and a \texttt{GAP} algorithm (\ref{appendix: code}) to compute the occurring permutation modules in the Appendix.

The approach fails for BMW algebras, the third main example for cellularly stratified algebras in \cite{HHKP}, since the appearing Hecke algebras are not subalgebras of BMW algebras. However, this is satisfied for $q$-Brauer algebras, another deformation of Brauer algebras, and there is hope that the theory applies in this case. 

\section{Preliminaries}\label{sec: preliminaries}
Let $k$ be an algebraically closed field, $r$ a natural number and $A$ an associative $k$-algebra. We denote the symmetric group on $r$ letters by $\Sigma_r$; its Iwahori-Hecke algebra is denoted by $\mathcal{H}_{k,q}(\Sigma_r)$, for some unit $q \in k$. Let $h$ be the smallest integer such that $\sum\limits_{i=0}^{h-1} q^i =0$. If $q=1$, then $h=\text{\textnormal{char}}k$. If $q$ is an $n$\th root of unity, then $h=n$.  

\begin{df}[\hspace*{-3pt}\cite{HHKP}, Definition 2.1]
An algebra $A$  is called \emph{cellularly stratified} if the following holds.
\begin{enumerate}
\item For each $l=0,...,r$ there is a cellular algebra $B_l$ and a vector space $V_l$ such that $A= \bigoplus\limits_{l=0}^r B_l \otimes_k V_l \otimes_k V_l$ as a vector space, respecting within each layer the multiplication of $A$, i.e. $A$ is an iterated inflation of the cellular algebras $B_l$ along the vector spaces $V_l$ as defined in \cite{KXinflation}. 
\item For all $l=0,...,r$ there are elements $u_l,v_l \in V_l\setminus\{0\}$ such that $e_l:=1_{B_l}\otimes u_l \otimes v_l$ is an idempotent and $e_le_{l'} = e_l = e_{l'}e_l$ for all $l' \geq l$. 
\end{enumerate}
The tuple $(B_0,V_0,...,B_r,V_r)$ is called \emph{stratification data} of $A$. 
\end{df}

It follows from the first part of the definition that $A$ is cellular with a chain of two-sided ideals $$0 \subseteq J_1 \subseteq ... \subseteq J_r = A$$ such that $J_l/J_{l-1} = B_l \otimes_k V_l \otimes_k V_l$ as a non-unital algebra (\hspace*{-3pt}\cite[Proposition 3.1 and § 3.2]{KXinflation}) which we call the \emph{$l$\th layer} of $A$, and $J_l = Ae_lA$ (\hspace*{-3pt}\cite[Lemma 2.2]{HHKP}). The product of $x \in J_l \setminus J_{l-1}$ and $y \in J_{l'}\setminus J_{l'-1}$ lies in $J_t$, where $t = \min\{l,l'\}$ by \cite[§ 3.2]{KXinflation}. 

\begin{rk}
If $A$ is cellularly stratified and the input algebra $B_l$ is isomorphic to a subalgebra of $e_lAe_l$, then $b \in B_l$ can be regarded as an element $b \otimes u_l \otimes v_l \in A$, where $u_l$ and $v_l$ are the vectors from the definition of $e_l$. In this case, we have $$be_l=(b \otimes u_l \otimes v_l)(1 \otimes u_l \otimes v_l) =b \varphi(v_l,u_l) \otimes u_l \otimes v_l = b \otimes u_l \otimes v_l = \varphi(v_l,u_l)b \otimes u_l \otimes v_l = e_l b$$ since $\varphi(v_l,u_l)=1$ by a remark on page 5 of \cite{HHKP}, where $\varphi$ is the bilinear form $\varphi: V_l \otimes_k V_l \to B_l$ defining the multiplication in the inflation, cf. \cite[§ 3.1]{KXinflation}.
\end{rk}

\begin{prop}\label{prop: eAe stratified}
Let $A$ be cellularly stratified such that $B_l$ is isomorphic to a subalgebra of $e_lAe_l$ for some $l < r$. If $B_n \subseteq B_{n+1}$ for all $n < l$, the algebra $e_lAe_l$ is cellularly stratified with stratification data $(B_0,V_0^l, ... ,B_l,V_l^l),$ where $V_n^l \subseteq V_n$ is a subspace such that $e_n \in B_n \otimes_k V_n^l \otimes_k V_n^l$, i.e. $u_n, v_n \in V_n^l$.
\end{prop}

\begin{proof}
Let $A = \bigoplus\limits_{n=0}^r B_n \otimes_k V_n \otimes_k V_n$. Then 
$$\begin{aligned} 
e_lAe_l  & = e_l(\bigoplus\limits_{n=0}^r B_n \otimes_k V_n \otimes_k V_n)e_l  \\ 
& = e_l(\bigoplus\limits_{n=0}^l B_n \otimes_k V_n \otimes_k V_n)e_l \\
& \subseteq B_l \oplus (\bigoplus\limits_{n=0}^{l-1} B_n \otimes_k V_n \otimes_k V_n) \end{aligned}$$ where the inclusion holds up to the isomorphism $B_l \simeq B_l \otimes_k \langle u_l \rangle_k \otimes_k \langle v_l\rangle_k$. Hence we have $e_lAe_l = \bigoplus\limits_{n=0}^l B_n^l \otimes_k V_n^l \otimes_k V_n^l$ for some $B_n^l \subseteq B_n, V_n^l \subseteq V_n$. It is $e_le_ne_l = e_n = 1_{B_n} \otimes u_n \otimes v_n \in B_n^l \otimes_k V_n^l \otimes_k V_n^l$  for all $n \leq l$. Since $B_n \subseteq B_l$, we have $be_l = e_lb$ for $b \in B_n$ by the  remark above, so $ b \otimes u_n \otimes v_n  = be_n = be_le_ne_l = e_lbe_ne_l  \in B_n^l \otimes_k V_n^l \otimes_k V_n^l$. Hence $B_n^l = B_n$.   
\end{proof}

Our main example of a cellularly stratified algebra will be the partition algebra $P_k(r,\delta)$. It is defined as follows. 

\begin{df}\label{def: partition algebra} Let $k$ be an algebraically closed field of arbitrary characteristic. Let $r \in \mathbb{N}$ and $\delta \in k$. The partition algebra $P_k(r,\delta)$ is the algebra with basis given by all set partitions of $\{1,...,r,1',...,r'\}$. To each set partition, we associate an equivalence class of diagrams consisting of two rows of $r$ dots each. Two dots $a$ and $b$ are connected via a path $a-...-b$ if and only if they belong to the same part of the set partition. Two diagrams are equivalent, if they correspond to the same set partition. 

\begin{ex}
The set partition $\{\{1,2'\},\{2,1',3'\},\{3,4'\},\{4\}\}$ corresponds to the diagram \begin{minipage}{2.6cm}$\xysmall{\bullet \tra[dr] & \bullet \tra[dl] & \bullet \tra[dr] & \bullet \\ \bullet \arcd[rr] & \bullet & \bullet & \bullet}$\end{minipage} with path $2 - 1' - 3'$ as well as to the diagram \begin{minipage}{2.6cm}$\xysmall{\bullet \tra[dr] & \bullet \tra[dr] & \bullet \tra[dr] & \bullet \\ \bullet \arcd[rr] & \bullet & \bullet & \bullet}$\end{minipage} with path $2 - 3' - 1'$, and the diagrams are equivalent to each other. 
\end{ex} 

We choose to write all diagrams such that the paths are ordered decreasingly with respect to the order $r > r-1 > ... > 1 > 1' > 2' > ... > n'$, like in the first diagram of the above example. Multiplication is given by concatenation of diagrams. Parts which are not connected to either top or bottom row (called \emph{inner circles}) are replaced by a factor $\delta \in k$.

\begin{ex}
Let $x=$\begin{minipage}{26mm}$\xysmall{\bullet & \bullet \arcd[rr] & \bullet \tra[dll] & \bullet \\ \bullet \arcu[rrr] & \bullet \tra[r] & \bullet & \bullet}$\end{minipage} and $y=$\begin{minipage}{26mm}$\xysmall{\bullet \tra[d] & \bullet \tra[r] & \bullet & \bullet \tra[dll] \\ \bullet & \bullet \tra[r] & \bullet \tra[r] & \bullet}$\end{minipage} in $P_k(4,\delta)$ then we have $xy=$\begin{minipage}{26mm}$\xysmall{\bullet & \bullet \arcd[rr] & \bullet \tra[dll] & \bullet \\ \bullet \arcu[rrr]\tra[d] & \bullet \tra[r]\tra[d] & \bullet\tra[d] & \bullet\tra[d] \\\bullet \tra[d] & \bullet \tra[r] & \bullet & \bullet \tra[dll] \\ \bullet & \bullet \tra[r] & \bullet \tra[r] & \bullet}$\end{minipage} $=\delta \cdot$\begin{minipage}{26mm}$\xysmall{\bullet & \bullet \arcd[rr] & \bullet \tra[dll] & \bullet \\ \bullet \tra[r] & \bullet \tra[r] & \bullet \tra[r] & \bullet}$
\end{minipage}
\end{ex}
 For further details (in a non-diagrammatic setting), see for example \cite{Xi}.
\end{df}

Note that multiplication of diagrams can decrease the number of \emph{propagating parts}, i.e. parts connecting top and bottom row, but never increase the number of propagating parts. 

A diagram consisting of only one row with $r$ dots and arbitrary connections is called \emph{partial diagram}. We have to distinguish certain parts from others; we say they are \emph{labelled} and write the dots as empty circles $\circ$ instead of dots $\bullet$. When we complete a partial diagram to a full diagram with two rows of dots, the labelled parts become propagating, i.e. they are connected to the other row.  We count the parts from left to right, according to the leftmost dot of each part. 
Let $V_n$ be the vector space with basis all partial diagrams with exactly $n$ labelled parts (and possibly further unlabelled parts). 
For example, $\xymatrixcolsep{10pt}\xymatrix{\bullet \arcu[rr] & \circ & \bullet \tra[r] & \bullet & \circ \tra[r] & \circ & \bullet}$ is a basis element of $V_2$, with $r=7$; the labelled singleton $\circ$ is the first labelled part, the part $\circ - \circ$ is the second.
We write $\top(d)$ to denote the top row of a diagram $d \in P_k(r,\delta)$ and $\bottom(d)$ for its bottom row. The permutation induced by the propagating parts is denoted by $\Pi(d)$. It is well-defined by the convention to connect labelled top and bottom row parts via their respective leftmost dots. \\

 If $\delta \neq 0$, the partition algebra is cellularly stratified by \cite[Proposition 2.6]{HHKP} with stratification data $(k, V_0, k, V_1, k\Sigma_2, V_2, ..., k\Sigma_r, V_r)$.
The idempotents are given by 
$$e_0 := \frac{1}{\delta} \cdot \begin{minipage}{3.7cm} \xysmall{ \bullet^1 \tra[r]  & \bullet \tra[r] & ...  \tra[r]  & \bullet \tra[r] & \bullet^r \\  \bullet_{1'} \tra[r] & \bullet \tra[r] & ...  \tra[r]  & \bullet \tra[r] & \bullet_{r'} } \end{minipage},\text{ } e_n:= \begin{minipage}[c]{3.7cm} \xysmall{\bullet^1 \tra[d] & ... & \bullet \tra[d] & \bullet^n \tra[r] \tra[d] &  ...  \tra[r]  & \bullet^r \\ \bullet_{1'} & ... & \bullet & \bullet_{n'} \tra[r] &  ...  \tra[r] & \bullet_{r'}} \end{minipage}\text{ for }n\geq 1.$$

For $0 \leq l \leq r$, there is an algebra isomorphism $P_k(l,\delta) \to e_l P_k(r,\delta)e_l$ given by attaching $r-l$ dots to the right of both top and bottom row and connecting the new dots to the rightmost dots of top and bottom row respectively of the original diagram.

The partition algebra $P_k(r,\delta)$ contains the Brauer algebra $B_k(r,\delta)$ and the group algebra $k\Sigma_r$ of the symmetric group $\Sigma_r$ as subalgebras. The Brauer algebra is the subalgebra with basis given by all diagrams where each dot is connected to exactly one other dot. We call such a connection \emph{(horizontal) arc} if it connects two dots within the same row. A permutation $\sigma \in \Sigma_r$ corresponds to the diagram connecting the $i$\th dot of the top row to the $\sigma(i)$\th dot of the bottom row.  

\subsection{Setup}\label{subsec: setup}
Let $A$ be cellularly stratified with stratification data $(B_0,V_0,...,B_r,V_r)$ where the $B_l$ are isomorphic to group algebras of symmetric groups or their Iwahori-Hecke algebras, such that for each $l\in \{0,...,r\}$ we have an embedding ${B_l \hookrightarrow e_lAe_l}$ of algebras. This is satisfied for Brauer algebras and partition algebras, but not for BMW-algebras, the third main example of cellularly stratified algebras in \cite{HHKP}. However, it is satisfied for another deformation of Brauer algebras: the {$q$-Brauer} algebras defined by Wenzl in \cite{Wenzl}. We choose as cell modules for the cellular algebras $B_l$ the dual Specht modules $S_\lambda$. \\

We need two types of induction and two types of restriction functors, namely 
$$\begin{aligned}
ind_l:B_l-\mod &\to A-\mod &\quad Ind_l:B_l-\mod &\to A-\mod \\
M &\mapsto Ae_l \tensorel M  &\quad M &\mapsto  Ae_l \tensorover{B_l} M\\
\hfill
res_l:A-\mod &\to B_l-\mod &\quad Res_l:A-\mod &\to B_l-\mod\\
N &\mapsto e_l(A/J_{l-1}) \tensorover{A} N    &\quad N &\mapsto e_lA \tensorover{A} N \simeq e_lN \\
\end{aligned}$$ 
where $J_l$ denotes the two-sided ideal $Ae_lA$ and $e_l(A/J_{l-1})$ is a short notation for $e_lA/e_lJ_{l-1}$.

\begin{rk} $Ae_l$ has a right $B_l$-module structure because we assumed $B_l$ to be isomorphic to a subalgebra of $e_lAe_l$. 
  Any $B_l$-module $M$ has an $e_lAe_l$-module structure via the quotient map $e_lAe_l \twoheadrightarrow e_l(A/J_{l-1})e_l \simeq B_l$, cf. \cite[Lemma 2.3]{HHKP}.
\end{rk}
 
Let $N \in A-\mod$. We call the left-ideal $(J_n/J_{n-1}) \tensorover{A} N$ the \emph{$n\th$ layer} of $N$.
The functor $ind_l$ sends a $B_l$-module $M$ to an $A$-module living in the $l\th$ layer, i.e. $ind_lM = (J_l/J_{l-1})\tensorover{A}ind_lM = (J_l/J_{l-1})e_l \tensorel M$. This is explained in the beginning of Subsection \ref{subsec: functors}.
We call this functor \emph{layer induction}. 

The induction functor $Ind_l$ sends a $B_l$-module $M$ to an $A$-module with non-zero action of $J_{l-1}$, i.e. $Ind_lM$ lives in all layers $n$ with $n\leq l$.
 
While $res_l$ removes the lower layers (with $n < l$) of the $A$-module $N$, $Res_l$ keeps all layers of the module. 

\subsection{Properties of the Functors}\label{subsec: functors}

 For each $B_l$-module $X$, we have $X \simeq B_l \tensorover{B_l} X \simeq B_l \tensorel X $, where $e_lAe_l$ acts on both $X$ and $B_l$ via $e_lAe_l \twoheadrightarrow e_l(A/J_{l-1})e_l \simeq B_l$ (\hspace*{-3pt}\cite[Lemma 2.3]{HHKP}). Thus, the layer induction $ind_l$ corresponds to the functor $G_l := Ae_l \tensorel B_l \tensorel -$, defined in \cite{HHKP}.
 Hence, we can apply \cite[Lemma 3.4]{HHKP} to get an isomorphism  $ind_lX \simeq (A/J_{l-1})e_l \tensorel X$ of $A$-modules. We will make extensive use of the isomorphisms $$ind_lX\simeq G_lX \simeq (A/J_{l-1})e_l \tensorel X \simeq (A/J_{l-1})e_l \tensorover{B_l} X$$ without special mention. 

\begin{prop}[\hspace*{-3pt}\cite{HHKP}, Propositions 4.1 - 4.3; Corollary 7.4; Propositions 8.1 and 8.2]\label{prop: HHKP properties ind} 
The functor $ind_l$ has the following properties.
\begin{enumerate}
\item It is exact. \label{ind property exact}
\item The set $\{ ind_lC | l=0,...,r; C \text{ cell module of } B_l\}$ is a complete set of cell modules for $A$. \label{ind property cells induced}
\item $\Hom_{B_l}(X,Y)\simeq \Hom_A(ind_lX,ind_lY)$ for all $X,Y \in B_l-\mod$.  \label{ind property Hom}
\item $\Ext_A^i(M,N) \simeq \Ext_{A/J_l}^i(M,N)$ for all $i>0$ and $M,N \in A/J_l-\mod$. \label{ind property Ext A/J}
\item $\Ext_{B_l}^j(X,Y) \simeq \Ext_A^j(ind_lX,ind_lY)$ for all $j\geq 0$ and $X,Y \in B_l-\mod$. \label{ind property Ext B}
\end{enumerate}
If $l<m$ then 
\begin{enumerate}\setcounter{enumi}{5}
\item $\Hom_A(ind_lX,ind_mY)=0$ for all $X \in B_l-\mod, Y \in B_m-\mod$. \label{ind property Hom directed}
\item $\Ext_A^i(ind_lX,ind_mY)=0$ for all $i\geq 1$ and $X \in B_l-\mod, Y \in B_m-\mod$. \label{ind property Ext directed}
\end{enumerate} 
\end{prop}

The induction $Ind_l$ is not exact in general and does not send cell modules to cell modules. 
However, we will give sufficient conditions for $Ind_l$ to send cell filtered modules to cell filtered modules in Section \ref{sec: permutation modules}. Theorem \ref{thm: relative projective} will tell us that, under additional conditions, $Ind_l$ sends relative projective modules to relative projective modules, cf. Definition \ref{def: rel proj}.

The following properties of the functors are straightforward calculations. The layer restriction $res_l$ is right-exact, but in general not exact. It is left adjoint to $\Hom_{B_l}(e_l(A/J_{l-1}),-)$ and left inverse to both $ind_l$ and $Ind_l$. 
The restriction $Res_l$ is exact, since $e_lA$ is projective as right $A$-module. It is left adjoint to $\Hom_{B_l}(e_lA,-)$ and right adjoint to $Ind_l$, i.e. we have a triple $(Ind_l, Res_l, \Hom_{B_l}(e_lA,-))$ of adjoint functors. Furthermore, $Res_l$ is left inverse to $ind_l$, but in general not to $Ind_l$; the layers added by $Ind_l$ are not removed by $Res_l$. 

For example, if $A$ is the Brauer algebra $B_\mathbb{C}(3,\delta)$ with $\delta \neq 0$ and $l=3$, and $X$ is the trivial $\mathbb{C}\Sigma_3$-module $\mathbb{C}$, then $e_3J_1e_3 = J_1 = Ae_1A$, which consists of all linear combinations of Brauer diagrams with exactly one horizontal arc per row. The left $\mathbb{C}\Sigma_3$-module $Res_3Ind_3 \mathbb{C}$ contains $Ae_1A \tensorover{\mathbb{C}\Sigma_3} \mathbb{C}$ which has a basis $$\left\lbrace \bigg[\begin{minipage}{1.4cm} \xymini{\bullet \tra[d] & \bullet \arcd[r] & \bullet \\ \bullet & \bullet \arcu[r] & \bullet} \end{minipage} \bigg], \bigg[\begin{minipage}{1.4cm} \xymini{\bullet \arcd[rr] & \bullet \tra[dl] & \bullet \\ \bullet & \bullet \arcu[r] & \bullet} \end{minipage} \bigg], \bigg[\begin{minipage}{1.4cm}\xymini{\bullet  \arcd[r] & \bullet & \bullet \tra[dll] \\ \bullet & \bullet \arcu[r] & \bullet}\end{minipage} \bigg] \right\rbrace,$$ where the brackets denote residue classes containing all three bottom row configurations. In particular, $Ae_1A \tensorover{\mathbb{C}\Sigma_3} \mathbb{C}$ is non-zero and not isomorphic to $X$. 

\begin{prop} If $X$ is a cell module of $A$, then $res_lX$ is a cell module of $B_l$ or zero. 
\end{prop}

\begin{proof}
Let $X$ be a cell module of $A$. By Proposition \ref{prop: HHKP properties ind}, part \emph{(\ref{ind property cells induced})}, we have $X \simeq ind_nS_\nu$ for some $1 \leq n \leq r$, where $S_\nu$ is a dual Specht module in $B_n-\mod$. This implies $res_lX \simeq res_lind_nS_\nu \simeq 
e_l(A/J_{l-1}) \tensorover{A} (A/J_{n-1})e_n \tensorover{e_nAe_n} S_\nu \simeq  e_l(A/J_m)e_n \tensorover{e_nAe_n} S_\nu$, where $m= \max\{l-1,n-1\}$.
If $n < l$, then $e_n \in J_m=J_{l-1}$ and if $n > l$, then $e_l \in J_m=J_{n-1}$. So, in both cases we have $res_lX=0$. For $n=l$, we have $res_lind_lS_\nu\simeq e_l(A/J_{l-1}) \tensorover{A} (A/J_{l-1})e_l \tensorel S_\nu\simeq e_l(A/J_{l-1})e_l\tensorover{B_l} S_\nu \simeq S_\nu$. Thus, the layer restriction of a cell module from the same layer is a cell module, while cell modules from other layers vanish under restriction. 
\end{proof}

\subsection{Further Definitions and Notation}\label{subsec: def permutation modules etc}
Let $\Lambda_r:=\{(l,\lambda) | 0\leq l\leq r, \lambda \vdash l'\}$, where $l'$ is the index of the symmetric group related to $B_l$ and $\lambda \vdash l'$ means that $\lambda$ is a partition of $l'$. We define an order $\prec$ on $\Lambda_r$ by setting  $$(n,\nu) \prec (l,\lambda) \Leftrightarrow n\geq l \text{ and if } n=l \text{ then } \nu \leq \lambda \text{ in the dominance order}. $$  
Let $(l,\lambda) \in \Lambda_r$ and let $M^\lambda$ be the corresponding permutation module in $B_l-\mod$.

\begin{df}
 We call the $A$-module $M(l,\lambda):=Ind_lM^\lambda$ {\em permutation module} for $A$.
\end{df}

Let $\Theta := \{\Theta(l,\lambda):= ind_l S_\lambda \,|\, (l,\lambda) \in \Lambda_r\}$ denote the set of cell modules. The category of $A$-modules with a cell filtration, i.e.\phantom{s}modules $M$ admitting a chain of submodules $M=M_n \supset M_{n-1} \supset ... \supset M_1 \supset M_0 = 0$ such that the subquotients $M_m/M_{m-1}$ are isomorphic to cell modules, is denoted by $\mathcal{F}(\Theta)$. The category of $B_l$-modules admitting a filtration by dual Specht modules is denoted by $\mathcal{F}_l(S)$. 

\begin{df}[\hspace*{-3pt}\cite{HHKP}, Definition 11.2]\label{def: rel proj}
Let $M, M' \in \mathcal{F}(\Theta)$. We say that $M$ is \emph{relative projective in $\mathcal{F}(\Theta)$}, if \begin{center}
$\Ext_A^1(M,N)=0$ for all $N \in \mathcal{F}(\Theta)$.\end{center}
 $M \in \mathcal{F}(\Theta)$ is the \emph{relative projective cover} of $M'$, if $M$ is minimal with respect to the property that there is an epimorphism $f: M \twoheadrightarrow M'$ with $\ker f \in \mathcal{F}(\Theta)$.
 \end{df}

\section{Young Modules}\label{sec: permutation modules}
In this section, we define Young modules as direct summands of permutation modules, following the definitions given for Brauer algebras by Hartmann and Paget, \cite{HP}. This allows us to extend the results of James for group algebras of symmetric groups to cellularly stratified algebras whose input algebras are isomorphic to group algebras of symmetric groups or their Hecke algebras. 

\begin{thm}\label{def: Y}
Let $A$ be a cellularly stratified algebra with input algebras isomorphic to group algebras of symmetric groups or their Hecke algebras. Assume that $e_lAe_l \simeq B_l \oplus e_lJ_{l-1}e_l$ as $(B_l,B_l)$-bimodules. 
Then $Ind_lM^\lambda$ has a unique direct summand with quotient isomorphic to $ind_l Y^\lambda$. 
\end{thm}

\begin{proof} 
It is well-known that the $B_l$-permutation module $M^\lambda$ decomposes into a direct sum of indecomposable Young modules $Y^\mu$ with multiplicities $a_\mu$, where $a_\lambda=1$ and $a_\mu \neq 0$ implies $\mu \leq \lambda$ (\hspace*{-3pt}\cite[Theorem 3.1]{Jtriv}). Therefore, we have  $Ind_lM^\lambda = \bigoplus\limits_{(l,\mu) \in \Lambda_r} (Ind_lY^\mu)^{a_\mu}$. Decompose $Ind_lY^\lambda$ further into a direct sum of indecomposables $Y_i$ for $i=1,...,s$.

\begin{claim} $Ind_lY^\lambda$ has a direct summand with quotient isomorphic to $ind_lY^\lambda$. \end{claim} 

Let $\pi_i:Ind_lY^\lambda \twoheadrightarrow Y_i$ be the projection onto $Y_i$ and $\iota_i: Y_i \hookrightarrow Ind_lY^\lambda$ the inclusion of $Y_i$.
The functor $Res_l$ is exact, so applying it to the composition $\iota_i \circ \pi_i$ gives maps $$\xymatrix{e_lA \tensorover{A} Ae_l \tensorover{B_l} Y^\lambda \ar[r]^-{e_lA \otimes \pi_i} & e_lA \tensorover{A} Y_i \ar[r]^-{e_lA \otimes \iota_i} & e_lA \tensorover{A} Ae_l \tensorover{B_l} Y^\lambda}.$$
By assumption, we have a decomposition $e_lAe_l \simeq B_l \oplus e_lJ_{l-1}e_l$ of right $B_l$-modules. Thus the homomorphism $Res_l(\iota_i \circ \pi_i):=(e_lA \otimes \iota_i)\circ (e_lA \otimes \pi_i) $ is given by a matrix, where the top left entry is an endomorphism  $f_i \in \End_{B_l}(Y^\lambda)$. This gives a commutative diagram
$$\xymatrix{ Ae_l \tensorover{B_l} Y^\lambda \ar@{->>}[r]^-{\pi_i} \ar[d]^-{Res_l} & Y_i \ar@{^{(}->}[r]^-{\iota_i} & Ae_l \tensorover{B_l} Y^\lambda \ar[d]^-{Res_l} \\
e_lAe_l \tensorover{B_l} Y^\lambda \ar[rr]^-{Res_l(\iota_i \circ \pi_i)} \ar[d]^\wr && e_lAe_l\tensorover{B_l} Y^\lambda \ar[d]^\wr \\
Y^\lambda \oplus (e_lJ_{l-1}e_l \tensorover{B_l} Y^\lambda) \ar[rr]^-{\tiny\begin{pmatrix} f_i & g_1 \\ g_2 & g_3
\end{pmatrix}} && Y^\lambda \oplus (e_l J_{l-1}e_l \tensorover{B_l} Y^\lambda)} $$ 

Let $y \in Y^\lambda$. Let $\pi_i(e_l \otimes y) = e_l \otimes x \, +\, lower \, terms$ for some $x \in Y^\lambda$. By \emph{lower terms} we mean terms of the form $e_lje_l \otimes z$ with $j \in J_{l-1}$ and $z \in Y^\lambda$. 

The commutativity of the above diagram says that, up to isomorphism, we have
$Res_l(\iota_i \circ \pi_i)(e_l\otimes y) = \begin{pmatrix}
f_i & g_1 \\ g_2 & g_3
\end{pmatrix}\begin{pmatrix}
y \\ 0
\end{pmatrix} = \begin{pmatrix}
f_i(y) \\ lower \, terms
\end{pmatrix}$, so we have $\pi_i(e_l \otimes y) = e_l \otimes f_i(y) \, + \, lower \, terms$. 

The identity on $Ind_lY^\lambda$ is $\sum\limits_{i=1}^s \pi_i$, so $$e_l \otimes y = \sum\limits_{i=1}^s \pi_i(e_l \otimes y) = \sum\limits_{i=1}^s (e_l \otimes f_i(y)) \, + \, lower \, terms$$ for any $y \in Y^\lambda$. Since there are no lower terms on the left hand side, they vanish on the right hand side and we have $e_l \otimes y = \sum\limits_{i=1}^s (e_l \otimes f_i(y))$. Hence $\sum\limits_{i=1}^s f_i(y) = y$, i.e. $\sum\limits_{i=1}^s f_i$ is the identity on $Y^\lambda$.  

Let $i \neq j$. Then $\pi_i \iota_j \pi_j = 0$, so for any $y \in Y^\lambda$ we have $0 = \pi_i \iota_j \pi_j(e_l \otimes y)= e_l \otimes f_if_j(y) \, + \, lower \, terms$. Therefore, $f_if_j =0$.
$Y^\lambda$ is finite dimensional and indecomposable, so $\End_{B_l}(Y^\lambda)$ is local. Thus  for all $i$, either $f_i$ or $1-f_i$ is a unit. To show that at least one $f_i$ is a unit, assume that $f_1,...,f_{s-1}$ are non-units. Then $\prod\limits_{i=1}^{s-1}(1-f_i) = 1 - f_1 - ... - f_{s-1} = f_s$ is a unit. 
We now assume without loss of generality that $f_1$ is a unit, in particular surjective. 

Let 
$$\begin{aligned} 
\varphi:& Ind_lY^\lambda &\longrightarrow & \, ind_lY^\lambda \\ 
& e_l \otimes y & \longmapsto & \, e_l \otimes  y 
\end{aligned}$$ 
and $\varphi':= \varphi \circ \iota_1 \circ \pi_1$ its restriction to $Y_1$. Then $$\varphi'(e_l \otimes y) = \varphi(e_l \otimes f_1(y)\, + \, lower \, terms) = e_l \otimes f_1(y),$$ since for $j \in J_{l-1}$ and $z \in Y^\lambda$, $\varphi(je_l \otimes z)=je_l \otimes z = 0 \in ind_lY^\lambda$. 
 Surjectivity of $f_1$ implies that the $A$-homomorphism $\varphi'$ is surjective, so $ind_lY^\lambda$ is a quotient of $Y_1$.

\begin{claim} $Y_1$ is the only summand of $Ind_lY^\lambda$ with quotient isomorphic to $ind_lY^\lambda$. \end{claim}

Suppose there is another summand $Y_2$ of $Ind_lY^\lambda$ such that there is an epimorphism $\psi:Ind_lY^\lambda \twoheadrightarrow ind_lY^\lambda$ with $\psi(Y_2)=ind_lY^\lambda$ and $\psi(Y_j)=0$ for all $j \neq 2$. By tensor-hom adjunction, $\psi$ is an element in $\Hom_A(Ind_lY^\lambda, ind_lY^\lambda)\simeq \Hom_{B_l}(Y^\lambda, \Hom_A(Ae_l, Ae_l \tensorover{e_lAe_l} Y^\lambda)) \simeq \Hom_{B_l}(Y^\lambda, e_lAe_l\tensorover{e_lAe_l}Y^\lambda) \simeq \End_{B_l}(Y^\lambda)$, so $\psi$ is given by $$\psi(e_l \otimes y)= e_l \otimes g(y)$$ for some $g \in \End_{B_{l}}(Y^\lambda)$. 
For $j \in J_{l-1}$ and $z \in Y^\lambda$, we have $$\psi(je_l \otimes z) = je_l \otimes g(z) = 0 \in ind_lY^\lambda.$$
The surjectivity of $\psi$ provides the existence of a preimage $v=\sum\limits_{i=1}^s (a_ie_l \otimes y_i) \in Y_2$ of $e_l  \otimes y \in ind_l Y^\lambda$ with $a_i \in A$ and $y_i \in Y^\lambda$ for all $i$. 
Since $e_lAe_l$ decomposes into $B_l \oplus e_lJ_{l-1}e_l$ as $(B_l,B_l)$-bimodule, we can write any element $e_lae_l \in e_lAe_l$ as $b + e_lje_l$ with $b \in B_l$ and $j \in J_{l-1}$. Thus $e_lv= e_l(\sum\limits_i a_ie_l \otimes y_i) = \sum\limits_i e_la_ie_l \otimes y_i = e_l \otimes w + \, lower \, terms$ for some $w \in Y^\lambda$. So $\psi$ sends $e_lv$ to $$\psi(e_lv)=\psi(e_l \otimes w + \, lower \, terms\,) = e_l \otimes g(w).$$
On the other hand, $$\psi(e_lv)=e_l\psi(v)=e_l(e_l\otimes y) = e_l \otimes y,$$ so $g(w)=y \neq 0$, 
hence $w\neq 0$. 
But $e_lv \in Y_2$ and $$\varphi'(e_lv)=\varphi'(e_l \otimes w + \, lower \, terms) = e_l \otimes f_1(w) \neq 0$$ since $w \neq 0$ and $f_1$ is a unit, in particular injective. So $\varphi'(Y_2)\neq 0$, which contradicts the definition of $\varphi'$.

\begin{claim} There is no summand of $Ind_lY^\mu$ with quotient $ind_lY^\lambda$ for $\mu \neq \lambda$. \end{claim}

Assume there is a direct summand $Y^\mu$ of $M^\lambda$ with $\mu > \lambda$ such that $ind_l Y^\lambda$ is a quotient of $Ind_l Y^\mu$. 
An arbitrary homomorphism $\Phi:Ind_lY^\mu \to ind_lY^\lambda$ is given by $\Phi(e_l \otimes y) = e_l \otimes \varphi(y)$ for some $\varphi \in \Hom_{B_l}(Y^\mu, Y^\lambda)$ by the adjunction $\Hom_A(Ind_lY^\mu, ind_lY^\lambda) \simeq \Hom_{B_l}(Y^\mu,Y^\lambda)$. $\Phi$ is surjective only if $\varphi$ is surjective\footnote{Assume there is $w \in Y^\lambda$ such that $\varphi(y) \neq w$ for all $y \in Y^\mu$. Let $\sum (a_ie_l \otimes y_i)$ be an arbitrary element of $Ind_lY^\mu$ and suppose that $\Phi(\sum (a_ie_l \otimes y_i)) = \sum (a_ie_l \otimes \varphi(y_i)) = e_l \otimes w$. Then $a_i = e_l$ for all $i$ and $\sum \varphi(y_i)=\varphi(\sum y_i) = w. \, \lightning$}. 

The rest of the proof can be copied from \cite{HP} in case $B_l = k\Sigma_{l'}$. We give here a similar proof for Iwahori-Hecke algebras $\mathcal{H}:=\mathcal{H}_{k,q}(\Sigma_l)$, inspired by the one for group algebras of symmetric groups, using notation from \cite{DJ}.

Suppose there is an epimorphism $\varphi: Y^\mu \twoheadrightarrow Y^\lambda$, which we extend to an epimorphism $\hat{\varphi}: M^\mu \to Y^\lambda$ such that $\hat{\varphi}$ is zero on all summands other than $Y^\mu$, i.e. $\hat{\varphi}$ is the projection from $M^\mu$ onto the direct summand $Y^\mu$, followed by the map $\varphi$. Recall (e.g. from \cite{DJ})
that $\mathcal{H}$ is generated by elements $T_\pi$, $\pi \in \Sigma_l$ and $M^\mu = \mathcal{H}x_\mu$, where $x_\mu = \sum\limits_{\omega \in \Sigma_\mu} T_\omega$. For $y_{\lambda'} = \sum\limits_{\omega \in \Sigma_{\lambda'}} (-q)^{l(\omega)}T_\omega$, where $l$ is the length function on symmetric group elements and $\lambda'$ is the conjugate of the partition $\lambda$, we have that $y_{\lambda'}T_\pi x_\mu \neq 0$ implies $\lambda=\lambda'' \geq \mu$ by \cite[Lemma 4.1]{DJ}. So for $\mu > \lambda$, we have $y_{\lambda'}M^\mu =0$. 
 Then $0=\hat{\varphi}(0)=\hat{\varphi}(y_{\lambda'}M^\mu )= y_{\lambda'}\hat{\varphi}(M^\mu) = y_{\lambda'}Y^\lambda$. But $y_{\lambda'}Y^\lambda$ contains the generator $y_{\lambda'}T_{w_\lambda}x_{\lambda} = z_\lambda$ of $S^\lambda$, in particular $y_{\lambda'}Y^\lambda \neq 0$.
 
This concludes the proof of Theorem \ref{def: Y}.
\end{proof}

\begin{df}
We denote the unique summand of $Ind_lY^\lambda$ with quotient $ind_lY^\lambda$ constructed above by $Y(l,\lambda)$, in analogy to \cite{HP}, and call it \textit{Young module} for $A$ with respect to $(l,\lambda)\in \Lambda_r$.
\end{df}

We now collect conditions for a Young module $Y(m,\mu)$ to appear as a summand of $M(l,\lambda)$. They generalise the conditions from \cite[Lemmas 17 and 18]{HP} for $A=B_k(r,\delta)$. The fact that these are the only direct summands of permutation modules is our main result (Theorem \ref{main thm}) and will be proven using results from the next Section. 

\begin{lem}\label{lemma: HP17}
 If $(l,\lambda),(m,\mu)\in \Lambda_r$ with $l < m$, then $Y(m,\mu)$ does not appear as a summand of $M(l,\lambda)$.
\end{lem} 

\begin{proof}
$Ind_l$ is left adjoint to $Res_l$, so $$\begin{aligned}\Hom_A(Ind_lM^\lambda, ind_mY^\mu)  & \simeq \Hom_{B_l}(M^\lambda, Res_lind_mY^\mu) \\  
& \simeq \Hom_{B_l}(M^\lambda, e_l(A/J_{m-1})e_m \tensorover{e_mAe_m} Y^\mu). \end{aligned}$$ For $l < m$, $e_l \in J_{m-1}$, so $Res_lind_mY^\mu=0$. Thus, there cannot be a non-zero map $$Ind_lM^\lambda \to Y(m,\mu)$$ since it would extend to a non-zero map $Ind_lM^\lambda \to ind_mY^\mu$.
\end{proof}

\begin{lem}\label{lemma: HP18}
 If $(l,\lambda),(l,\kappa) \in \Lambda_r$, then $Y(l,\lambda)$ occurs as a direct summand of $M(l,\kappa)$ if and only if $Y^\lambda$ is a direct summand of $M^{\kappa}$. This can only occur if $\lambda \geq \kappa$. 
\end{lem}

\begin{proof}
If $Y^\lambda$ is a direct summand of $M^\kappa$, then $Y(l,\lambda)$, as a direct summand of $Ind_lY^\lambda$, is a direct summand of $Ind_lM^\kappa = M(l,\kappa)$. 

If $Y(l,\lambda)$ is a direct summand of $M(l,\kappa)$ and $M^\kappa = \bigoplus (Y^\mu)^{a_\mu}$, then $Y(l,\lambda)$ is a summand of $Ind_lY^\mu$ for some $\mu$. 

It follows from Theorem \ref{def: Y}, Claim 3, that $\mu = \lambda$, so $Y^\lambda$ is a direct summand of $M^\kappa$.
\end{proof}

\begin{cor}\label{cor: Young modules not isomorphic}
If $(l,\lambda) \neq (l,\kappa)$, then $Y(l,\lambda) \nsimeq Y(l,\kappa)$. 
\end{cor}

\begin{proof}
Let $(l,\lambda) \neq (l,\kappa)$. Then $Y^\lambda \not\simeq Y^\kappa$, see for example \cite[Section 7.6]{Martin}, so $Ind_lY^\lambda \nsimeq Ind_lY^\kappa$ since otherwise $res_lInd_lY^\lambda \simeq Y^\lambda$ would be isomorphic to $res_lInd_lY^\kappa \simeq Y^\kappa$. Assume that $Y(l,\lambda) \simeq Y(l,\kappa)$. Then $Y(l,\kappa)$ is a direct summand of $M(l,\lambda)$ and by Lemma \ref{lemma: HP18}, $Y^\kappa$ is a direct summand of $M^\lambda$. So $Ind_lY^\kappa$ is a summand of $M(l,\lambda)$ and has a summand $Y(l,\kappa)$ with quotient $ind_lY^\kappa$. But $Y(l,\kappa)$ is isomorphic to $Y(l,\lambda)$ with quotient $ind_lY^\lambda$, so $Ind_lY^\kappa$ has a direct summand with quotient isomorphic to $ind_lY^\lambda$ and $\kappa \neq \lambda$. This contradicts Claim 3 from Theorem \ref{def: Y}.
\end{proof}

\section{Properties}\label{sec: properties pm}
Each Young module $Y(l,\lambda)$ is a direct summand of the permutation module $M(l,\lambda)=Ind_lM^\lambda$ by definition. In this section, we show that the indecomposable direct summands of permutation modules are exactly the Young modules, as in the symmetric group case. The results extend the results on Brauer algebras stated in \cite{HP} to our setup. 
 
We give conditions under which the permutation modules for our cellularly stratified algebra $A$ admit a cell filtration in Subsection \ref{subsec: cell filtrations}. In Subsection \ref{subsec: rel proj}, we show  that permutation modules are relative projective in the subcategory $\mathcal{F}(\Theta)$ of cell filtered $A$-modules, provided a further condition is satisfied. Then the Young module $Y(l,\lambda)$ is  the relative projective cover of the cell module $\Theta(l,\lambda):= ind_lS_\lambda$ (Theorem \ref{thm: relative projective}). As a corollary of this, we recover a result about Schur-Weyl duality from \cite{HHKP} in Subsection \ref{subsec: SW duality}. Finally, we can prove Theorem \ref{main thm}, the decomposition of the permutation module $M(l,\lambda)$ into a direct sum of Young modules $Y(l,\lambda)$, in Subsection \ref{subsec: decomposition}.

A crucial point in the study of a category $\mathcal{F}(\Delta)$ of $\Delta$-filtered $A$-modules is that it is closed under direct summands if the set $\Delta$ with ordered index set $(I,\leq)$ forms a \emph{standard system}\footnote{cf. \cite[Section 3]{DR} or \cite[Definition 10.1]{HHKP}}, i.e. for all $l,m \in I$ 
\begin{itemize}
\item $\End_A(\Delta(l))$ is a division ring.
\item $\Hom_A(\Delta(l),\Delta(m))\neq 0$ implies $l \geq m$.
\item $\Ext_A^1(\Delta(l),\Delta(m)) \neq 0$ implies $l > m$.
\end{itemize}
The statement follows from \cite[Theorem 2]{Ringel}. 

\begin{lem}\label{lemma: standard system} 
Let $A$ be as defined in Subsection \ref{subsec: setup}. Let $\text{\textnormal{char}}k = p \in \mathbb{Z}_{\geq 0}\setminus\{2,3\}$ if the input algebras are group algebras of symmetric groups and let $h \geq 4$ if the input algebras $B_l$ are isomorphic to Hecke algebras $\mathcal{H}_{k,q}(\Sigma_l)$.
Then the cell modules $\Theta$ of $A$ form a standard system with respect to the order $\prec$ defined in Subsection \ref{subsec: def permutation modules etc}.
\end{lem}

\begin{proof}
Dual Specht modules for symmetric groups form a standard system by \cite[Proposition 4.2.1]{HN} and \cite[Corollary 13.17]{JamesLectureNotes}. Dual Specht modules for Iwahori-Hecke algebras of symmetric groups form a standard system by \cite[Proposition 4.2.1]{HN} and \cite[Exercise 4.11]{Mathas}. The statement follows from \cite[Theorem 10.2 (a)]{HHKP}. 
\end{proof}

\begin{assumptions}\label{assumptions}
We give names to the following assumptions that we make on $A$ in order to prove the desired properties for permutation modules and Young modules. Furthermore, we often assume that $\text{\textnormal{char}} k \in \mathbb{Z}_{\geq 0}\setminus\{2,3\}$ (or $h  \geq 4$, in case the $B_l$ are Iwahori-Hecke algebras) to be able to use Lemma \ref{lemma: standard system}. 

Let $A$ be as defined in Subsection \ref{subsec: setup} and let $n \leq l$. 
\begin{enumerate}[(I)]
\item \begin{enumerate}
\item $e_lAe_l \simeq B_l \oplus e_lJ_{l-1}e_l$ as $(B_l,B_l)$-bimodules and \label{assumptionpart: eAe}
\item $J_ne_l \simeq J_{n-1}e_l \oplus (J_n/J_{n-1})e_l$ as right $B_l$-modules.\label{assumptionpart: Je}
\end{enumerate} \label{assumption: J_ne_l}
\item $(J_n/J_{n-1})e_l \simeq (A/J_{n-1})e_n \tensorover{e_nAe_n} e_n(A/J_{n-1})e_l$ as right $B_l$-modules. \label{assumption: aufplustern}
\item Layer-removing restriction to $B_n-\mod$ of a permutation module from layer $l$ is dual Specht filtered: $$res_nInd_lM^\lambda \simeq e_n(A/J_{n-1})e_l \tensorover{B_l} M^\lambda \in \mathcal{F}_n(S)$$ \label{assumption: restriction of pm}
\item Classical restriction to $B_l-\mod$ of a cell module from layer $n$ is dual Specht filtered: $$Res_lind_nS_\nu \simeq e_l(A/J_{n-1})e_n \tensorover{B_n} S_\nu \in \mathcal{F}_l(S)$$ \label{assumption: restriction of cell modules}
\end{enumerate}
\end{assumptions}

\begin{rk} Assumption (\ref{assumptionpart: eAe}) is the assumption we made in Theorem \ref{def: Y} in order to define Young modules for $A$.   

Assumption (\ref{assumption: restriction of cell modules}) implies that for any $X \in \mathcal{F}_n(S)$, $Res_lind_nX \in \mathcal{F}_l(S)$: The functor $ind_n$ is exact and sends dual Specht modules to cell modules, so $ind_nX$ has a cell filtration. $Res_l$ is exact, so $Res_lind_nX$ has a filtration by modules of the form $Res_lind_nS_\nu \in \mathcal{F}_l(S)$. The statement follows since $\mathcal{F}_l(S)$ is extension-closed.  
\end{rk}

\begin{lem}\label{lemma: replace assumption}
Instead of \textnormal{(\ref{assumption: aufplustern})}, we can assume \newline

\indent \textnormal{(II')} $(J_n/J_{n-1})e_l \simeq B_n \otimes_k V_n \otimes_k V_n^l$ as vector spaces. 

\end{lem}

\begin{proof}
By Proposition \ref{prop: eAe stratified}, the algebra $e_lAe_l$ is cellularly stratified with idempotents $e_n =1_{B_n} \otimes u_n \otimes v_n \in B_n \otimes_k V_n^l \otimes_k V_n^l \subseteq B_n \otimes_k V_n \otimes_k V_n$. Then $e_n(A/J_{n-1})e_l = e_n(e_lAe_l/e_lJ_{n-1}e_l)$ is free of rank $\dim V_n^l$ over $B_n$ by \cite[Proposition 3.5]{HHKP} and $ ind_n(e_n(A/J_{n-1})e_l)\simeq (A/J_{n-1})e_n \tensorover{B_n} e_n(A/J_{n-1})e_l \simeq \! \bigoplus\limits_{i=1}^{\dim V_n^l}\!(A/J_{n-1})e_n$ as left $A$-modules. 
Hence, $\dim (ind_n(e_n(A/J_{n-1})e_l)) = \dim ((A/J_{n-1})e_n) \cdot \dim V_n^l = \dim B_n \cdot \dim V_n \cdot \dim V_n^l$, since $(A/J_{n-1})e_n$ is free of rank $\dim V_n$ over $B_n$.

The multiplication map $$\begin{aligned}
(A/J_{n-1})e_n \tensorover{B_n} e_n(A/J_{n-1})e_l & \longrightarrow  (J_n/J_{n-1})e_l \\
(a+J_{n-1})e_n \otimes e_n(b+J_{n-1})e_l & \longmapsto  (ae_nb+J_{n-1})e_l \end{aligned}$$
is an epimorphism of $(A,B_l)$-bimodules and $\dim (ind_n(e_n(A/J_{n-1})e_l)) = \dim V_n^l \cdot \dim V_n \cdot \dim B_n = \dim ((J_n/J_{n-1})e_l)$ by (II'), so (\ref{assumption: aufplustern}) is satisfied.  
\end{proof}

\subsection{Cell filtrations}\label{subsec: cell filtrations}
\begin{thm} \label{thm: cell filtration} Assume that $A$ satisfies \textnormal{(\ref{assumption: J_ne_l}),(\ref{assumption: aufplustern})} and \textnormal{(\ref{assumption: restriction of pm})}. Then the permutation module $M(l,\lambda)$ has a filtration by cell modules. \newline 
If, in addition, $\text{\textnormal{char}}k \in \mathbb{Z}_{\geq 0}\setminus\{2,3\}$ or $h \geq 4$, then the direct summands of $Ind_lM^\lambda$ have cell filtrations. 
\end{thm}

\begin{proof}
$A = J_r \supset J_{r-1} \supset ... \supset J_1 \supset J_0 = 0$ is a filtration of $A$ (with quotients isomorphic to $B_n \otimes_k V_n \otimes_k V_n$), so we have short exact sequences $$0 \to J_{n-1} \to J_n \to J_n/J_{n-1} \to 0$$ of $(A,A)$-bimodules for $1 \leq n \leq r$. Application of the exact restriction functor $-\tensorover{A} Ae_l$ gives exact sequences $$0 \to  J_{n-1}e_l \to J_ne_l \to (J_n/J_{n-1})e_l \to 0$$ of $(A,e_lAe_l)$-bimodules for $n \leq l$, which are split exact as sequences of right $B_l$-modules by assumption (\ref{assumptionpart: Je}). Hence, we get exact sequences 
$$0\to J_{n-1}e_l \tensorover{B_l} M^\lambda \to J_ne_l \tensorover{B_l} M^\lambda \to (J_n/J_{n-1})e_l \tensorover{B_l} M^\lambda \to 0$$ of left $A$-modules, which give rise to a filtration $$Ae_l \tensorover{B_l} M^\lambda \supset J_{l-1}e_l \tensorover{B_l} M^\lambda \supset ...\supset J_1e_l \tensorover{B_l} M^\lambda \supset 0$$  of $M(l,\lambda)=Ind_l M^\lambda$ 
with quotients $M^n(l,\lambda):= (J_n/J_{n-1})e_l \tensorover{B_l} M^\lambda$, the $n\th$ layer of $M(l,\lambda)$. Assumption (\ref{assumption: aufplustern}) gives
$$\begin{aligned} M^n(l,\lambda) & \simeq ind_n(e_n(A/J_{n-1})e_l) \tensorover{B_l} M^\lambda \\ & \simeq 
ind_n(e_n(A/J_{n-1})e_l \tensorover{B_l} M^\lambda) \\ & \simeq ind_n(res_nInd_lM^\lambda). \end{aligned}$$
                
By assumption (\ref{assumption: restriction of pm}), $res_nInd_lM^\lambda \in \mathcal{F}_n(S)$. The functor $ind_n$ is exact and sends dual Specht modules to cell modules by Proposition \ref{prop: HHKP properties ind}, so $ M^n(l,\lambda) \in \mathcal{F}(\Theta)$ for all $1 \leq n \leq l$, in particular $M(l,\lambda)=M^l(l,\lambda) \in \mathcal{F}(\Theta)$. \\

If $\text{\textnormal{char}}k$ is different from $2$ and $3$, then the cell modules of $A$ form a standard system by Lemma \ref{lemma: standard system}. In this case, $\mathcal{F}(\Theta)$ is closed under direct summands by \cite[Theorem 2]{Ringel}, so all direct summands of $Ind_lM^\lambda$, in particular the Young modules $Y(l,\lambda)$, admit cell filtrations.
\end{proof}

\subsection{Relative projectivity}\label{subsec: rel proj}
An important property of the permutation modules ${M^\lambda \in B_l-\mod}$ is their relative projectivity in the category $\mathcal{F}_l(S)$, as shown by Hemmer and Nakano in \cite[Proposition 4.1.1]{HN}, in case $h\geq 4$. This property is translated to the permutation modules $M(l,\lambda)$ of $A$, in case the conditions (\ref{assumption: J_ne_l}) to (\ref{assumption: restriction of cell modules}) are satisfied. Furthermore, the Young modules are relative projective covers of the cell modules.

\begin{thm} \label{thm: relative projective} Assume that $A$ satisfies \textnormal{(\ref{assumption: J_ne_l})} to \textnormal{(\ref{assumption: restriction of cell modules})}. Then the permutation module $Ind_lM^\lambda$ is relative projective in $\mathcal{F}(\Theta)$. If, in addition, $\text{\textnormal{char}}k \in \mathbb{Z}_{\geq 0}\setminus\{2,3\}$ (or $h \geq 4$), then all direct summands of $Ind_lM^\lambda$ are relative projective in $\mathcal{F}(\Theta)$. Furthermore, $Y(l,\lambda)$ is the relative projective cover of $\Theta(l,\lambda)$ in the category $\mathcal{F}(\Theta)$ of cell filtered modules.  
\end{thm}

\begin{proof}
By Theorem \ref{thm: cell filtration}, $M(l,\lambda)$ and all its direct summands (provided $\text{\textnormal{char}}k \neq 2,3$ or $h \geq 4$) are in $\mathcal{F}(\Theta)$ if $A$ satisfies conditions (\ref{assumption: J_ne_l}) to (\ref{assumption: restriction of pm}). We have to show that $\Ext^1_A(M(l,\lambda),X)=0$ for all $X \in \mathcal{F}(X)$. Let $X \in \mathcal{F}(\Theta)$ and let $$(\ast): \, 0 \to X \to Y \to Ind_lM^\lambda \to 0$$ be a short exact sequence in $\Ext^1_A(M(l,\lambda),X)$.

Apply the exact functor $Res_l$ on $(\ast)$ to get a short exact sequence $$(\ast \ast): \, 0 \to e_lX \to e_lY \to e_lAe_l \tensorover{B_l} M^\lambda \to 0$$ in $B_l-\mod$. Now we apply the left exact functor $\Hom_{B_l}(M^\lambda,-)$ to get a long exact sequence
$$\xymatrixcolsep{12pt}\xymatrix{0 \ar[r] & \Hom_{B_l}(M^\lambda,e_lX) \ar[r] & \Hom_{B_l}(M^\lambda,e_lY) \ar[r] &  \Hom_{B_l}(M^\lambda,e_lAe_l\tensorover{B_l}M^\lambda) \ar@{->}`r[d]`[dlll]`[ddll][ddll] \\ &&&&& \\
& \Ext^1_{B_l}(M^\lambda,e_lX) \ar[r] &...&&} $$

It follows from assumption (\ref{assumption: restriction of cell modules}) and the exactness of $Res_l$ that $e_lX \in \mathcal{F}_l(S)$ for $X \in \mathcal{F}(\Theta)$. Since $M^\lambda$ is relative projective in $\mathcal{F}_l(S)$, we get $\Ext^1_{B_l}(M^\lambda,e_lX)=0$, in particular we get a short exact sequence 
$$0 \to \Hom_{B_l}(M^\lambda,e_lX) \to \Hom_{B_l}(M^\lambda,e_lY) \to \Hom_{B_l}(M^\lambda,e_lAe_l\tensorover{B_l}M^\lambda) \to 0$$ which is isomorphic to the short exact sequence
$$(\diamond): \, 0 \to \Hom_{A}(Ind_lM^\lambda,X) \to \Hom_{A}(Ind_lM^\lambda,Y)  \overset{f}{\rightarrow} \End_{A}(Ind_lM^\lambda) \to 0 $$ since $Res_l$ is right adjoint to $Ind_l$.   

Consider $$ (\ast): \,\xymatrix{ 0 \ar[r] & X \ar[r] & Y \ar[r]^\alpha & Ind_lM^\lambda \ar[r] & 0 \\
&&& Ind_lM^\lambda \ar@{=}[u] \ar@{-->}[ul]^{\exists\beta}}$$
then $\beta$ exists (such that the diagram commutes) by surjectivity of the map $f$  in $(\diamond)$. This shows that $(\ast)$ splits and so $\Ext^1_A(M(l,\lambda),X)=0$. In particular, $M(l,\lambda)$ is relative projective in $\mathcal{F}(\Theta)$.\\

Now let $Z$ be a direct summand of $M(l,\lambda)$ with $\pi: Ind_lM^\lambda \to Z$ the projection onto $Z$ and $\iota: Z \to Ind_lM^\lambda$ the inclusion of $Z$ into $M(l,\lambda)$. 
With the same strategy as above, applied to the short exact sequence $$(\star): \, 0 \to X \to Y \to Z \to 0, $$ we see that the map $\Hom_A(Ind_lM^\lambda,Y)\rightarrow \Hom_A(Ind_lM^\lambda,Z)$ is surjective, which provides the existence of a map $f: Ind_lM^\lambda \to Y$ such that $\pi = gf$: 
$$\xymatrix{0 \ar[r] & X \ar[r] & Y \ar[r]^g & Z \ar[r]\ar@<1ex>@{^{(}->}[d]^\iota & 0 \\
 &&& Ind_lM^\lambda  \ar@{-->}[ul]^{\exists f} \ar@<1ex>@{>>}[u]^\pi }$$ But $\pi\iota=id_Z$, so $gf\iota=id_Z$ and $f\iota$ is right inverse to $g$. Therefore, the sequence $(\star)$ splits and $\Ext^1_A(Z,X)=0$, so all direct summands of $Ind_lM^\lambda$ are relative projective in $\mathcal{F}(\Theta)$. \\

In order to prove that $Y(l,\lambda)$ is the relative projective cover of $\Theta(l,\lambda)$, we have to show that there is an epimorphism $$\Psi:Y(l,\lambda) \twoheadrightarrow \Theta(l,\lambda)$$ with $\ker(\Psi) \in \mathcal{F}(\Theta)$ and  that $Y(l,\lambda)$ is minimal with respect to this property. Once we have established the epimorphism, the minimality condition is immediately satisfied since $Y(l,\lambda)$ is indecomposable, and then $Y(l,\lambda)$ is a relative projective cover of $\Theta(l,\lambda)$.  
 
The $B_l$-module $Y^\lambda$ has a dual Specht filtration with top quotient $S_\lambda$, so the kernel of the map $Y^\lambda \twoheadrightarrow S_\lambda$ lies in $\mathcal{F}_l(S)$. The functor $ind_l$ is exact and sends dual Specht modules to cell modules, so the kernel of the epimorphism $$\psi:ind_lY^\lambda \twoheadrightarrow ind_lS_\lambda = \Theta(l,\lambda)$$ has a cell filtration. 

Recall from the proof of Theorem \ref{def: Y} that there is an epimorphism $$\phi: Y(l,\lambda) \overset{\iota}{\hookrightarrow} Ind_lY^\lambda \overset{\varphi}{\to} ind_lY^\lambda.$$ Consider the commutative diagram 
$$\xymatrix{&  					& 0 \ar[d] \\
	   		&  					& \ker\phi \ar[d] 							&  \\
0 \ar[r]	& \ker\Psi \ar[r]& Y(l,\lambda) \ar[r]^{\Psi = \phi\psi}\ar[d]^\phi	& \Theta(l,\lambda) \ar[r] \ar@{=}[d]	& 0\\
0\ar[r]		& \ker\psi\ar[r]		& ind_l Y^\lambda \ar[d]\ar[r]^\psi		& \Theta(l,\lambda) \ar[r] & 0\\
			& 					& 0											&  } $$
with $\ker\psi$, $Y(l,\lambda)$, $ind_lY^\lambda$ and $\Theta(l,\lambda)$ in $\mathcal{F}(\Theta)$.  
The composition $$\ker \Psi \to Y(l,\lambda) \overset{\phi}{\to} ind_lY^\lambda \overset{\psi}{\to} \Theta(l,\lambda)$$ is zero, so the universal property of the kernel of $\psi$ provides a unique morphism $\ker \Psi \to \ker\psi$, with kernel $K$, making the diagram	 
$$\xymatrix{& 0	\ar[d]					& 0 \ar[d] \\
	   		& K	\ar[r]^\sim	\ar[d]			& \ker\phi \ar[d] 							 \\
0 \ar[r]	& \ker\Psi \ar[r]\ar[d]	& Y(l,\lambda) \ar[r]^{\Psi}\ar[d]^\phi	& \Theta(l,\lambda)\ar[r]\ar@{=}[d]& 0\\
0\ar[r]		& \ker\psi\ar[r]\ar[d]				& ind_l Y^\lambda \ar[d]\ar[r]^\psi			& \Theta(l,\lambda) \ar[r] 		& 0\\
			& 0					& 0											 } $$
commutative. The map $K \to \ker\phi$ is given by the universal property of the kernel of $\phi$ and is an isomorphism by the snake lemma. The snake lemma also asserts surjectivity of the map $\ker\Psi \to \ker\psi$. 

Thus, we have a short exact sequence $$0 \to \ker\phi \to \ker\Psi \to \ker\psi \to 0$$ with $\ker\psi \in \mathcal{F}(\Theta)$. If we can show that $\ker\phi = \ker \varphi\iota \in \mathcal{F}(\Theta)$, then $\ker\Psi \in \mathcal{F}(\Theta)$ since $\mathcal{F}(\Theta)$ is extension-closed. 

Consider the commutative diagram 
$$\xymatrix{0 \ar[r] & \ker\varphi\iota \ar[r] & Y(l,\lambda) \ar[r]^{\varphi\iota} \ar@{^{(}->}[d]^\iota & ind_lY^\lambda \ar[r] \ar@{=}[d] & 0 \\
0 \ar[r]  & \ker\varphi \ar[r] & Ind_lY^\lambda \ar[r]^\varphi & ind_lY^\lambda \ar[r] & 0}$$
We have $\iota(\ker\varphi\iota) \subseteq \ker\varphi$, so $\iota$ restricts to $\ker\varphi\iota \to \ker\varphi$.

Now, we consider the commutative diagram 
$$\xymatrix{0 \ar[r] & \ker\varphi\iota \ar[r] & Y(l,\lambda) \ar[r]^{\varphi\iota}  & ind_lY^\lambda \ar[r] \ar@{=}[d] & 0 \\
0 \ar[r] & \ker\varphi \ar[r] & Ind_lY^\lambda \ar[r]^\varphi \ar@{->>}[u]^\pi & ind_lY^\lambda \ar[r] & 0}$$
where $\pi$ is the projection from $Ind_lY^\lambda$ onto its summand $Y(l,\lambda)$. We see that $\pi(\ker\varphi) \subseteq \ker\varphi\iota$, so $\pi$ restricts to $\ker\varphi \to \ker\varphi\iota$.

$$\xymatrix{0 \ar[r] & \ker\varphi\iota \ar[r]\ar@{^{(}->}[d]^\iota & Y(l,\lambda) \ar[r]^{\varphi\iota} \ar@{^{(}->}[d]^\iota  & ind_lY^\lambda \ar[r] \ar@{=}[d] & 0 \\
0 \ar[r] & \ker\varphi \ar[r]\ar@<1ex>@{->>}[u]^\pi & Ind_lY^\lambda \ar[r]^\varphi \ar@<1ex>@{->>}[u]^\pi & ind_lY^\lambda \ar[r] & 0}$$

In particular, $\ker\varphi\iota$ is a direct summand of $\ker\varphi = J_{l-1}e_l \tensorover{B_l} Y^\lambda$. By the proof of Theorem \ref{thm: cell filtration}, the module $J_{l-1}e_l \tensorover{B_l} M^\lambda$ has a cell filtration. By the assumption on the characteristic of the field, cell filtrations restrict to direct summands, so $\ker\varphi$ and $\ker\varphi\iota$ lie in $\mathcal{F}(\Theta)$. Since $\mathcal{F}(\Theta)$ is extension-closed, we get $\ker\Psi \in \mathcal{F}(\Theta)$ and so $Y(l,\lambda)$ is a relative projective cover of $\Theta(l,\lambda)$.   
\end{proof}

\begin{cor}[\hspace*{-3pt}\cite{HHKP}, Corollary 12.4] 
If $B_l$ is a group algebra of a symmetric group $\Sigma_{l'}$ for some $l' \in \mathbb{N}$, $\text{\textnormal{char}}k = p \in \mathbb{Z}_{\geq 0}\setminus\{2,3\}$, and $A$ satisfies \textnormal{(\ref{assumption: J_ne_l})} to \textnormal{(\ref{assumption: restriction of cell modules})}, then $Y(l,\lambda)$ is projective if and only if $\lambda$ is $p$-restricted.
\end{cor}

\subsection{Schur-Weyl duality}\label{subsec: SW duality}
In \cite{HHKP}, the Young modules $Y_{pr}(l,\lambda)$ of a cellularly stratified algebra $A$ are defined as the relative projective covers of the cell modules $\Theta(l,\lambda)$, in the case where the cell modules of the input algebras $B_l$ form standard systems. Since we assumed $B_l$ to be isomorphic to $k\Sigma_{l'}$ or $\mathcal{H}_{k,q}(\Sigma_{l'})$ for some $l' \in \mathbb{N}$ and $\text{char} k \in \mathbb{Z}_{\geq 0}\setminus\{2,3\}$, respectively $h \geq 4$, we are in this situation (Lemma \ref{lemma: standard system}). Therefore, we have the following corollary of Theorem \ref{thm: relative projective}.

\begin{cor}\label{cor: young modules coincide}
The Young modules $Y_{pr}(l,\lambda)$, defined abstractly in \textnormal{\cite{HHKP}}, coincide with the explicitly defined Young modules $Y(l,\lambda)$ of this article.
\end{cor} 

In particular, we are in the situation of Theorem 13.1 from \cite{HHKP}:
 
\begin{cor}
\label{thm: HHKP 13.1} 
Let $A$ be as defined in Subsection \ref{subsec: setup}, such that the assumptions\textnormal{(\ref{assumption: J_ne_l})} to \textnormal{(\ref{assumption: restriction of cell modules})} are satisfied and let $\text{\textnormal{char}}k \in \mathbb{Z}_{\geq 0}\setminus \{2,3\}$ (or $h\geq 4$). Then the following holds.
\begin{enumerate}
\item Each $M \in \mathcal{F}(\Theta)$ has well-defined filtration multiplicities.
\item The category $\mathcal{F}_A(\Theta)$ of cell filtered $A$-modules is equivalent, as exact category, to the category $\mathcal{F}_{\End_A(Y)}(\Delta)$ of standard filtered modules over the quasi-hereditary algebra $\End_A(Y)$, where $$Y=\bigoplus\limits_{(l,\lambda) \in \Lambda_r} Y(l,\lambda)^{n_{l,\lambda}}$$ and  $n_{l,\lambda}= \begin{cases} \dim L(l,\lambda) & \text{ if there is a simple module } L(l,\lambda) \\ 1 & \text{ otherwise.} \end{cases}$ 
\item There is a Schur-Weyl duality between $A$ and $\End_A(Y)$. In particular, we have $A = \End_{\End_A(Y)}(Y)$. 
\end{enumerate}
\end{cor}

\begin{rk}
The multiplicities $n_{l,\lambda}$ of the Young modules $Y(l,\lambda)$ in $Y$ are chosen to be minimal such that all Young modules appear at least once and such that the projective Young modules appear as often as they appear in $A$, i.e. such that there is a $D \in A-\mod$ with $Y=A\oplus D$. 
\end{rk}

\subsection{Decomposition of permutation modules}\label{subsec: decomposition}
Using the results of the previous subsections, we are finally able to prove that permutation modules for $A$ decompose into a direct sum of Young modules, just like permutation modules for $B_l$ decompose into direct sums of Young modules. 

\begin{thm}\label{main thm} Let $A$ be as defined in Subsection \ref{subsec: setup}, such that the assumptions\textnormal{(\ref{assumption: J_ne_l})} to \textnormal{(\ref{assumption: restriction of cell modules})} are satisfied and let $\text{\textnormal{char}}k \in \mathbb{Z}_{\geq 0}\setminus \{2,3\}$ (or $h\geq 4$). Let $(l,\lambda) \in \Lambda_r$. Then there is a decomposition $$Ind_lM^\lambda = \bigoplus\limits_{(m,\mu)\succeq (l,\lambda)} Y(m,\mu)^{a_{m,\mu}}$$ with non-negative integers $a_{m,\mu}$. Moreover, $a_{l,\lambda}=1$.  
\end{thm}

\begin{proof}
By Lemma \ref{lemma: standard system}, the set $\Theta$  forms a standard system. 
Corollary \ref{thm: HHKP 13.1} says that there is a quasi-hereditary algebra $C = \End_A(Y)$ such that the categories $\mathcal{F}_A(\Theta)$ of cell filtered $A$-modules and $\mathcal{F}_C(\Delta)$ of standard filtered $C$-modules are equivalent, which was first established in \cite{DR}. To prove this equivalence, Dlab and Ringel show that there is a one-to-one correspondence between the modules in the standard system $\{\Theta\}$ and the indecomposable relative projective modules in $\mathcal{F}(\Theta)$. By Theorem \ref{thm: relative projective}, the Young modules $Y(l,\lambda)$ are indecomposable relative projective. The one-to-one correspondence shows that these are all indecomposable relative projective $A$-modules, since for each $(l,\lambda) \in \Lambda_r$ there is exactly one Young module and exactly one cell module, and these are all cell modules, cf. Propostition \ref{prop: HHKP properties ind} part \emph{(\ref{ind property cells induced})}, Theorem \ref{def: Y} and Corollary \ref{cor: Young modules not isomorphic}. The algebra $C$ is quasi-hereditary, so the relative projective $C$-modules are exactly the projective $C$-modules, cf. \cite[Corollary 2]{Ringel}, and they correspond under the equivalence to the relative projective $A$-modules. Hence, the projective $C$-modules are indexed by $\Lambda_r$. 

The permutation module $M(l,\lambda)$ is relative projective in $\mathcal{F}(\Theta)$, so its image under the equivalence $\mathcal{F}(\Theta) \xrightarrow{\,\smash{\raisebox{-0.65ex}{\ensuremath{\sim}}}\,} \mathcal{F}(\Delta)$ is a projective $C$-module $P$. Let $P=\bigoplus\limits_{(n,\nu)\in  \Lambda_r} P(n,\nu)^{a_{n,\nu}}$ be a decomposition of $P$ into indecomposable modules. Sending $P(n,\nu)$ back to $\mathcal{F}(\Theta)$ through the equivalence, its image must be an indecomposable relative projective module $Y(m,\mu)$. Thus, $M(l,\lambda) =\bigoplus\limits_{(m,\mu)\in\Lambda_r} Y(m,\mu)^{a_{m,\mu}}$ for some non-negative integers $a_{m,\mu}$. $a_{l,\lambda}=1$ by definition of $Y(l,\lambda)$. Lemmas \ref{lemma: HP17} and \ref{lemma: HP18} show that we only have to sum over those Young modules $Y(m,\mu)$ with $(m,\mu) \succeq (l,\lambda)$.  
\end{proof}

\section{Applications}\label{sec: applications}
There are three main examples of cellularly stratified algebras in \cite{HHKP}: Brauer algebras, partition algebras and Birman-Murakami-Wenzl algebras (BMW algebras), a deformation of Brauer algebras. The results for Brauer algebras first appeared in \cite{HP}. With the theory from this article, we can recover their results, using less combinatorics specific to Brauer algebras but the more structural properties of cellularly stratified algebras, which have been introduced after the work of Hartmann and Paget on Brauer algebras appeared. We recover the results for Brauer algebras in Subsection \ref{sec: applications Brauer algebras}, thus providing new proofs. In Subsection \ref{sec: applications partition algebras}, we show that the results hold for partition algebras under certain additional assumptions. The theory fails for BMW algebras, since we need the cellular algebras $B_l=\mathcal{H}_{k,q}(\Sigma_{l'})$ to be subalgebras. However, the $q$-Brauer algebras, defined by Wenzl in \cite{Wenzl}, are another deformation of Brauer algebras which fit into this setting. They are cellularly stratified as shown by Nguyen in his PhD thesis \cite{DungPhD} and contain Hecke algebras as subalgebras. We do not prove that the $q$-Brauer algebras satisfy the assumptions in this article.

\subsection{Recovering Results for Brauer Algebras}\label{sec: applications Brauer algebras}
Let $A = B_k(r,\delta) \subseteq P_k(r,\delta)$ be the Brauer algebra on $r$ dots with $\delta \in k$. 
If $r$ is even, let $\delta \neq 0$. Then by \cite[Proposition 2.4]{HHKP}, $A$ is cellularly stratified with stratification data $$(k\Sigma_t,V_t, k\Sigma_{t+2},V_{t+2},...,k\Sigma_{r-2},V_{r-2},k\Sigma_r,V_r),$$ where $t=0$ if $r$ is even and $t=1$ if $r$ is odd, and $V_l$ is the vector space with basis consisting of partial diagrams with exactly $\frac{r-l}{2}$ horizontal arcs. The idempotents $e_l$ are defined as $e_l = \frac{1}{\delta^{\frac{r-l}{2}}}\cdot \begin{minipage}[c]{4cm} 
\xymini{\overset{1}{\bullet} \tra[d]  & ... & \overset{l}{\bullet} \tra[d] & \bullet \arcd[r] & \bullet & ... & \bullet \arcd[r] & \overset{r}{\bullet} \\ 
\bullet & ... & \bullet & \bullet \arcu[r] & \bullet & ... & \bullet \arcu[r] & \bullet } \end{minipage}$ for $\delta \neq 0$. For $\delta =0$ (and $r$ odd), we use
$e_l = \begin{minipage}[c]{4cm} 
\xymini{ \overset{1}{\bullet} \tra[d]  & ... & \bullet \tra[d] & \overset{l}{\bullet} \tra[drrrrr] & \bullet \ar@{-}@/_3pt/[r] & \bullet & ... & \bullet \ar@{-}@/_3pt/[r] & \overset{r}{\bullet} \\ 
\bullet & ... & \bullet &  \bullet \ar@{-}@/^3pt/[r] & \bullet & ... & \bullet \ar@{-}@/^3pt/[r] & \bullet & \bullet} \end{minipage}$. 

We want to recover the results from \cite{HP}, so we have to show that the Young modules defined here coincide with those defined in \cite{HP} as indecomposable submodules of $Ind_lY^\lambda$ with quotient $V_l \tensorover{k} Y^\lambda$. The module structure on $V_l \tensorover{k} X$ is defined as follows. Let $b \in B_k(r,\delta)$ be a basis element and let $v \otimes x \in V_l \tensorover{k} X$. Then $$b(v \otimes x) = (bv) \otimes \pi(b,v)x $$
where $bv$ is the partial diagram obtained by writing $b$ on top of $v$, identifying $\bottom(b)$ with $v$ and following the new connections in $\top(b)$, multiplying by $\delta$ for each closed loop. If the result is not in $V_l$, set $bv = 0$. The permutation $\pi(b,v)$ is given by the permutation of the free dots of $v$ in $bv$. 

\begin{ex} Let $b = \begin{minipage}{4cm}\xysmall{\bullet \tra[dr] & \bullet \arcd[rr] & \bullet \tra[dll] & \bullet \\ \bullet  & \bullet & \bullet \arcu[r] & \bullet }\end{minipage} \in B_k(4,\delta)$ and $v = \xysmall{\bullet & \bullet & \bullet \arcu[r] & \bullet} \in V_2$. Then $bv = \delta^{\# \text{closed loops}}\top\left(\begin{minipage}{4cm}\xysmall{\bullet \tra[dr] & \bullet \arcd[rr] & \bullet \tra[dll] & \bullet \\ \bullet  & \bullet & \bullet \arcu[r] & \bullet 
\\ \bullet \ar@{=}[u] & \bullet \ar@{=}[u] & \bullet \arcu[r]\ar@{=}[u] & \bullet \ar@{=}[u]} \end{minipage}\right) = \delta  \xysmall{\bullet & \bullet \arcu[rr] & \bullet & \bullet} $ and $\pi(b,v) = (1,2)$.
\end{ex}

\begin{prop}\label{prop: Brauer algebra}
For any $X \in k\Sigma_l-\mod$, there is an isomorphism $ind_lX \simeq  V_l \tensorover{k} X$ of $B_k(r,\delta)$-modules.  
\end{prop}

\begin{proof} Let $X \in k\Sigma_l-\mod$ and consider the map 
$$\begin{aligned} \varphi: & V_l \otimes_k X & \longrightarrow &\, (A/J_{l-2})e_l \tensorover{k\Sigma_l} X \\ 
& v \otimes x & \longmapsto &\text{ } (d^v + J_{l-2}) \otimes x, \end{aligned}$$  where $d^v$ is the diagram in $J_le_l \setminus J_{l-2}e_l$ with $\top(d^v)=v$ and non-crossing propagating lines\footnote{Since $d^v$ is in $J_le_l$, its bottom row is fixed: $l$ free dots followed by $\frac{r-l}{2}$ horizontal arcs sitting side by side.}.
Let $(ae_l + J_{l-2}) \otimes x \in (A/J_{l-2})e_l \tensorover{k\Sigma_l} X$, with $ae_l + J_{l-2}$ corresponding to  $b \otimes w \otimes v_l$ under the isomorphism $J_l/J_{l-2} \simeq k\Sigma_l \otimes_k V_l \otimes_k V_l$, i.e. $ae_l+J_{l-2} = d^wb + J_{l-2}$. Then $\varphi(w \otimes bx) = (d^w + J_{l-2}) \otimes bx = (d^wb + J_{l-2}) \otimes x = (ae_l + J_{l-2}) \otimes x$, so $\varphi$ is surjective. 
By \cite[Proposition 3.5]{HHKP}, $\dim((A/J_{l-2})e_l \tensorover{k\Sigma_l} X) = \dim(k\Sigma_l^{\dim V_l} \tensorover{k\Sigma_l} X) = \dim V_l \cdot \dim X = \dim(V_l \otimes_k X)$. Hence, $\varphi$ is bijective. 
To see that $\varphi$ is an isomorphism, we have to check that it is $A$-linear. Let $a \in A$ and $v \otimes x \in V_l \otimes_k X$. Then $$\begin{aligned} \varphi(a(v\otimes x)) & = \varphi(av \otimes \pi(a,v)x) \\ &= (d^{av} + J_{l-2}) \otimes \pi(a,v)x \\ & = (d^{av}\pi(a,v) + J_{l-2}) \otimes x \end{aligned}$$ and $$\begin{aligned}
a \varphi(v \otimes x) &= a((d^v + J_{l-2}) \otimes x) \\ & = (ad^v + J_{l-2}) \otimes x.\end{aligned}$$ 
If $ad^v \in J_{l-2}$ then $a\varphi(v \otimes x)=(ad^v + J_{l-2}) \otimes x =  0$. On the other hand, $ad^v \in J_{l-2}$ implies that $av$ has more than $\frac{r-l}{2}$ horizontal arcs, so $\varphi(a(v\otimes x)) = \varphi(av \otimes \pi(a,v)x) = \varphi(0) = 0$. 
If $ad^v$ has $l$ propagating lines, then $a \in J_m \setminus J_{l-2}$ for some $m \geq l$ and $m-l$ of the free dots\footnote{In this case, a free dot is a dot which does not belong to a horizontal arc.} of $\top(a)$ are bound by horizontal arcs in $ad^v$ since the product lies in $J_l$. The remaining $l$ free dots of $\top(a)$ are end points of propagating lines in $ad^v$. Therefore, the permutation of the propagating lines of $ad^v$ is $\pi(a,v)$. This shows $a\varphi(v\otimes x)= (ad^v + J_{l-2})\otimes x = (d^{av}\pi(a,v) + J_{l-2}) \otimes x = \varphi(a(v\otimes x))$ and $\varphi$ is $A$-linear. 
\end{proof}

\begin{cor}
The cell, Young and permutation modules defined here coincide with those defined in \textnormal{\cite{HP}}.
\end{cor}

It remains to verify that $A=B_k(r,\delta)$, with $\delta \neq 0$ if $r$ is even, satisfies the assumptions (\ref{assumption: J_ne_l}) to (\ref{assumption: restriction of cell modules}).  
Let $0 \leq n \leq l \leq r$. 

The decompositions $e_lAe_l \simeq k\Sigma_l \oplus e_lJ_{l-2}e_l$ and $J_ne_l \simeq J_{n-2}e_l \oplus (J_n/J_{n-2})e_l$ hold for vector spaces. The left (resp. right) action of $k\Sigma_l-\mod$ permutes the dots of the top (resp. bottom) row, but it never changes the amount of horizontal arcs, so assumption (\ref{assumption: J_ne_l}) is satisfied. Assumption (\ref{assumption: aufplustern}) holds by \cite[Lemma 4.3]{HK1}. By \cite[Lemma 4.2]{HK1}, $e_n(J_n/J_{n-2})e_l  \simeq k\tensorover{H \times k\Sigma_n} k\Sigma_l$, where $H := k(C_2 \wr \Sigma_\frac{l-n}{2})$. We get the following isomorphisms of $k\Sigma_n$-modules $$res_nInd_lM^\lambda \simeq e_n(J_n/J_{n-2})e_l \tensorover{k\Sigma_l} M^\lambda \simeq k \tensorover{H \times k\Sigma_n} k\Sigma_l \tensorover{k\Sigma_l} M^\lambda \simeq k \tensorover{H \times k\Sigma_n} k\Sigma_l \tensorover{k\Sigma_\lambda} k.$$ The last module is equal to a direct sum of $k\Sigma_n$-permutation modules $M^\nu$ by \cite[Lemma 4.5]{HK1}. Therefore, $res_nInd_lM^\lambda \in \mathcal{F}_n(S)$ and assumption (\ref{assumption: restriction of pm}) is satisfied.
The restriction of a cell module $ind_nS_\nu$ to $k\Sigma_l-\mod$, with $l\geq n$, is dual Specht filtered by \cite[Proposition 8]{P}, thus $A$ satisfies assumption (\ref{assumption: restriction of cell modules}). This gives a new proof for the following theorem. 

\begin{thm}[\hspace*{-3pt}\cite{HP}]\label{thm: Brauer algebras}Let $\text{\textnormal{char}}k \neq 2,3$. The Brauer algebra $B_k(r,\delta)$, with $\delta \neq 0$ if $r$ is even, has permutation modules $M(l,\lambda)$, which are a direct sum of indecomposable Young modules. The Young modules are the relative projective covers of the cell modules $ind_lS_\lambda$. Every module admitting a cell filtration has well-defined filtration multiplicities.  
\end{thm}

\subsection{New Results for Partition Algebras}\label{sec: applications partition algebras}
Now, let $A=P_k(r,\delta)$ be the partition algebra on $r$ dots with $\delta \in k \setminus \{0\}$. Then $A$ is cellularly stratified by \cite[Proposition 2.6]{HHKP}. The stratification data, as well as an isomorphism $P_k(l,\delta) \simeq e_lP_k(r,\delta)e_l$ for $0 \leq l \leq r$, was described in Section \ref{sec: preliminaries}. We use the following embedding of $k\Sigma_l$ into $P_k(r,\delta)$. Let $d(\pi)\in P_k(l,\delta)$ be the diagram describing the permutation $\pi \in \Sigma_l$, i.e. the dot $i$ in the top row is connected to the dot $\pi(i)$ in the bottom row, and these are all the connections. Then $d(\pi)$ becomes an element of $P_k(r,\delta)$ by attaching dots $l+1, ..., r$ to the right of the  top row and connecting all these new dots to the $l$\th dot of the top row. Do the same for the bottom row. This embedding agrees with the isomorphism $P_k(l,\delta)\simeq e_lP_k(r,\delta)e_l$ from Section \ref{sec: preliminaries}. 

\begin{ex}
Let $\pi = (1432) \in \Sigma_4$ and let $r=7$. Then $d(\pi)= \begin{minipage}[c]{2.5cm} \xysmall{\bullet \tra[drrr] & \bullet \tra[dl] & \bullet \tra[dl] & \bullet \tra[dl] \\ \bullet & \bullet & \bullet & \bullet}
\end{minipage}$ is clearly an element of $P_k(4,\delta)$. The corresponding element in $P_k(7,\delta)$ is \linebreak $\begin{minipage}{3cm}
\xysmall{\bullet \tra[drrr] & \bullet \tra[dl] & \bullet \tra[dl] & \bullet \tra[dl] \tra[r] & \bullet \tra[r] &  \bullet \tra[r] & \bullet \\ \bullet & \bullet & \bullet & \bullet \tra[r] & \bullet \tra[r] &  \bullet \tra[r] & \bullet}
\end{minipage}.$
\end{ex}  

In particular, for each $0 \leq l \leq r$, the input algebra $k\Sigma_l$ of the cellularly stratified structure is a subalgebra of $e_lAe_l$. 

It remains to show that $A$ satisfies conditions (\ref{assumption: J_ne_l}) to (\ref{assumption: restriction of cell modules}). Fix some $l$ between $0$ and $r$ and remember that $J_l$ denotes the two-sided ideal $Ae_lA$. Set $J_{-1}:=0$.

The left (resp. right) action of $k\Sigma_l$ on a partition diagram $d$ permutes the top (resp. bottom) row of $d$, but it never changes the size of a part of $d$. In particular, the number of propagating lines remains invariant under the $k\Sigma_l$-action and the decompositions from assumption (\ref{assumption: J_ne_l}) are indeed decompositions of $k\Sigma_l$-(bi)modules. 
 
For $0 \leq n \leq l$, the basis diagrams of $(J_n/J_{n-1})e_l$ have exactly $n$ propagating parts and the last $r-l+1$ dots of the bottom row belong to the same part. Hence, we have an isomorphism of vector spaces $(J_n/J_{n-1})e_l \simeq k\Sigma_n \otimes_k V_n \otimes_k V_n^l$, where $V_n$ is the vector space of partial diagrams with exactly $n$ labelled parts and $V_n^l$ is the subspace of $V_n$ where the last $r-l+1$ dots belong to the same part. This shows assumption (II$'$) is satisfied and thus, by Lemma \ref{lemma: replace assumption}, assumption (\ref{assumption: aufplustern}) is satisfied as well.

Assumption (\ref{assumption: restriction of cell modules}) holds by \cite[Theorem 1]{arxiv} in case $\text{\textnormal{char}} k > \lfloor\frac{r}{3}\rfloor$. The condition on the characteristic is sufficient, but potentially too strong, as explained in \cite{arxiv}.

We now prove that assumption (\ref{assumption: restriction of pm}) is satisfied.
Fix $0 \leq n \leq l \leq r$. When dealing with the size of a part in a partial diagram, we will from now on count the last $r-l+1$ dots as one. 
Let $v,w \in V_n^l$. We say that $v$ is equivalent to $w$, $v \sim w$, if and only if there is a $\pi \in \Sigma_l$ such that $v\pi = w$, where $v\pi $ is defined as follows. Write the diagram $\pi$ below $v$ and identify $\top(\pi)$ with $v$. Then $v\pi $ is the bottom row of this diagram, where a part is labelled if and only if it contains at least one labelled dot. 
In diagrams, this means that $v$ and $w$ are equivalent if and only if for each size, the number of labelled parts and the number of unlabelled parts of $v$ and $w$ coincide. Remember that the last $r-l+1$ dots count as one. 

For $v \in V_n^l$, we define $d_v$ to be the diagram in $P_k(r,\delta)$ with $\top(d_v)=\top(e_n) $, $\bottom(d_v)= v$ and $\Pi(d_v)=1_{k\Sigma_n}$. 
Let $b \in e_n(A/J_{n-1})e_l$ be a diagram with $\bottom(b) \sim v$. By definition, there is a $\pi \in \Sigma_l$ such that $\bottom(b)=v\pi$. Then $b = \Pi(b) \Pi(d_v \pi)^{-1}  d_v\pi$.
 Let $U_v$ be the $(k\Sigma_n,k\Sigma_l)$-bimodule generated by $d_v$. 

\begin{lem}[{\hspace*{-3pt}\cite[Lemma 1]{arxiv}}]\label{lemma: arxiv Uv}
 The $(k\Sigma_n,k\Sigma_l)$-bimodule $e_n(A/J_{n-1})e_l$ decomposes into $\bigoplus\limits_{v \in V_n^l/_\sim} U_v$.
\end{lem}

Fix a partial diagram $v \in V_n^l$ and set $d:=d_v$. Let $\alpha_i$ be the number of labelled parts of size $i$ and $\beta_i$ the number of unlabelled parts of size $i$ of $v$, where again the last $r-l+1$ dots count as one dot. Then $\sum\limits_i (\alpha_i \cdot i) + \sum\limits_i (\beta_i \cdot i) = l$ and $\sum\limits_i \alpha_i = n$. Without loss of generality, assume that the parts of $v$ are ordered as follows. The labelled parts are on the left hand side, the unlabelled parts on the right hand side. The parts are then ordered increasingly from left to right. 

Let $\mathcal{S}_i^j \subseteq \{1,...,l\}$ be the set of dots of $v$ belonging to the $j$\th labelled part of size $i$ and let $\mathcal{T}_i^j \subseteq \{1,...,l\}$ be the set of dots of $v$ belonging to the $j$\th unlabelled part of size $i$. Then $\prod_\alpha:= \prod\limits_{{i \geq 1,  \alpha_i \neq 0}}((\Sigma_{\mathcal{S}_i^1} \times ... \times \Sigma_{\mathcal{S}_i^{\alpha_i}}) \rtimes \Sigma_{\alpha_i})$ is the stabilizer subgroup of $\Sigma_l$ which stabilizes exactly the labelled parts of $v$. Similarly, the stabilizer subgroup of $\Sigma_l$ which stabilizes the unlabelled parts of $v$ is $\prod_\beta:= \prod\limits_{{i \geq 1, \beta_i \neq 0}}((\Sigma_{\mathcal{T}_i^1} \times ... \times \Sigma_{\mathcal{T}_i^{\beta_i}}) \rtimes \Sigma_{\beta_i})$. In particular, $\prod_\beta$ stabilizes $d$, while $\prod_\alpha$ permutes the propagating lines of $d$. Note that $\prod_\alpha \simeq \prod\limits_{i \geq 1, \alpha_i \neq 0} (\Sigma_i \wr \Sigma_{\alpha_i})$
 and $\prod_\beta \simeq \prod\limits_{i \geq 1, \beta_i \neq 0} (\Sigma_i \wr \Sigma_{\beta_i})$, where $\wr$ denotes the wreath product. Define a right-action of $\prod_\alpha \times \prod_\beta$ on $k\Sigma_n$ via $\eta \cdot \zeta := \eta \Pi(d\zeta)$ for $\eta \in \Sigma_n$ and $\zeta \in \prod_\alpha\times \prod_\beta$, i.e. $\prod_\alpha \times \prod_\beta$ acts on $k\Sigma_n$ via the canonical epimorphism $$\rho:  \prod\limits_{i \geq 1, \alpha_i \neq 0} (\Sigma_i \wr \Sigma_{\alpha_i}) \times \prod\limits_{i \geq 1, \beta_i \neq 0} (\Sigma_i \wr \Sigma_{\beta_i}) \twoheadrightarrow \Sigma_\alpha.$$ Then we can define the tensor product $k\Sigma_n \tensorover{k\prod_\alpha \times k\prod_\beta} k\Sigma_l$. 
 
\begin{lem}[{\hspace*{-3pt}\cite[Lemma 2]{arxiv}}]\label{lemma: tensor decomposition of Uv}
There is an isomorphism of $(k\Sigma_n,k\Sigma_l)$-bimodules $k\Sigma_n \tensorover{k\prod_\alpha \times k\prod_\beta} k\Sigma_l  \longrightarrow U_v $ given by $\eta \otimes \tau \longmapsto \eta d \tau$. 
\end{lem}

We want to understand the summands $k\Sigma_n \tensorover{k\prod_\alpha \times k\prod_\beta} k\Sigma_l \tensorover{k\Sigma_\lambda} k$ of $res_nInd_lM^\lambda = e_n(A/J_{n-1})e_l \tensorover{k\Sigma_l} k\Sigma_l \tensorover{k\Sigma_\lambda} k$ for a partition $\lambda$ of $l$. Fix double coset representatives $\pi_1,...,\pi_q$ of $(\prod_\alpha \times \prod_\beta)\backslash \Sigma_l /\Sigma_\lambda$. To each $\pi_i$, we attach a composition $\nu^i$ as follows. Set $\prod_{\nu^i} := (\prod_\alpha \times \prod_\beta) \cap \pi_i\Sigma_\lambda\pi_i^{-1}$. Then $\zeta \in \prod_\alpha \times \prod_\beta$ is in $\prod_{\nu^i}$ if and only if there is a $\vartheta \in \Sigma_\lambda$ such that $\zeta\pi_i = \pi_i\vartheta$. Since $\pi_i\Sigma_\lambda\pi_i^{-1}$ is isomorphic to $\Sigma_\lambda$, it is a Young subgroup of $\Sigma_l$, and $\prod_\alpha \times \prod_\beta$ is a direct product of wreath products of symmetric groups. Then the intersection $(\prod_\alpha \times \prod_\beta) \cap \pi\Sigma_\lambda\pi^{-1}$ is again a product of wreath products. The image of $\prod_{\nu^i}$ under the canonical epimorphism $\rho$ is a Young subgroup of $\Sigma_n$, which we denote by $\Sigma_{\nu^i}$. 

An example for $\Pi_{\nu^i}$, $\Sigma_{\nu^i}$ and a GAP-algorithm to compute them can be found in the appendix. 

\begin{prop}\label{prop: partition algebras} The left $k\Sigma_n$-module $k\Sigma_n \tensorover{k(\prod_\alpha \times \prod_\beta)} k\Sigma_l \tensorover{k\Sigma_\lambda} k$ is isomorphic to the direct sum $\bigoplus\limits_{i=1}^q (k\Sigma_n \tensorover{k\Sigma_{\nu^i}} k)$ of various permutation modules. In particular, it admits a filtration by dual $k\Sigma_n$-Specht modules.
\end{prop} 

\begin{proof}
We define a map $$\varphi: k\Sigma_n \tensorover{k(\prod_\alpha \times \prod_\beta)} k\Sigma_l \tensorover{k\Sigma_\lambda} k \longrightarrow \bigoplus\limits_{i=1}^q (k\Sigma_n \tensorover{k\Sigma_{\nu^i}} k)$$ as follows. Let $\eta \in \Sigma_n$ and $\tau \in \Sigma_l$ with $\tau = \zeta\pi_i \vartheta$ for some $\zeta \in \prod_\alpha\times \prod_\beta$ and $\vartheta\in \Sigma_\lambda$. Set $\varphi(\eta \otimes \tau \otimes 1) = (0,...,0,\eta \Pi(d\zeta) \otimes 1,0,...,0) =: (\eta \Pi(d\zeta) \otimes 1)^{(i)}$ with non-zero entry only in the $i\th$ summand. Extend this $k\Sigma_n$-linearly to get a $k\Sigma_n$-homomorphism. 

We have to show that this map is well-defined, that is we have to show that whenever two elements $\eta \otimes \tau \otimes 1$ and $\eta' \otimes \tau' \otimes 1$ are equivalent in $k\Sigma_n \tensorover{k(\prod_\alpha \times \prod_\beta)} k\Sigma_l \tensorover{k\Sigma_\lambda} k$, then their images are equivalent in $\bigoplus\limits_{i=1}^q (k\Sigma_n \tensorover{k\Sigma_{\nu^i}} k)$. 

Let $\eta \otimes \tau \otimes 1 = \eta' \otimes \tau' \otimes 1$ with $\eta, \eta' \in \Sigma_n$ and $\tau, \tau' \in \Sigma_l$ and let $\tau = \zeta \pi_i \vartheta$ and $\tau'=\zeta' \pi_j \vartheta'$. Since $\eta \otimes \tau \otimes 1 = \eta' \otimes \tau' \otimes 1$, we have $i=j$ and $\eta\Pi(d\zeta) = \eta'\Pi(d\zeta')$. It follows that $\varphi(\eta \otimes \tau \otimes 1) = \varphi(\eta \otimes \zeta \pi_i \vartheta \otimes 1) = (\eta\Pi(d\zeta) \otimes 1)^{(i)} = (\eta'\Pi(d\zeta')\otimes 1)^{(i)} = \varphi(\eta' \otimes \tau' \otimes 1) $, so $\varphi$ is well-defined.

The inverse is given by $$\psi: \bigoplus\limits_{i=1}^q (k\Sigma_n \tensorover{k\Sigma_{\nu^i}}k) \longrightarrow k\Sigma_n \tensorover{k(\prod_\alpha\times\prod_\beta)}k\Sigma_l \tensorover{k\Sigma_\lambda} k$$ with $\psi(\sum\limits_{i=1}^q (\eta_i \otimes 1)^{(i)}) = \sum\limits_{i=1}^q \eta_i \otimes \pi_i \otimes 1$ for $\eta_i \in \Sigma_n$: 

$$(\psi \circ \varphi)(\eta  \otimes \zeta\pi_i \vartheta \otimes 1) = \psi((\eta \Pi(d\zeta) \otimes 1)^{(i)}) = \eta\Pi(d\zeta) \otimes \pi_i \otimes 1 = \eta \otimes \zeta\pi_i\vartheta \otimes 1$$ and $$(\varphi \circ \psi)((\eta \otimes 1)^{(i)}) = \varphi(\eta \otimes \pi_i \otimes 1) = (\eta\otimes 1)^{(i)}$$
for $\eta \in \Sigma_n$, $\zeta \in \prod_\alpha \times \prod_\beta$ and $\vartheta \in \Sigma_\lambda$.  

It remains to show that $\psi$ is well-defined. Let $\eta, \eta' \in \Sigma_n$ such that $\eta \otimes 1$ and $\eta' \otimes 1$ are equivalent in $k\Sigma_n \tensorover{k\Sigma_{\nu^i}} k$ for some $i$. Then there is a $\xi \in \Sigma_{\nu^i}$ such that $\eta' = \eta\xi$. It follows that $\psi((\eta' \otimes 1)^{(i)}) = \eta' \otimes \pi_i \otimes 1 = \eta\xi \otimes \pi_i \otimes 1 = \eta \otimes \hat{\xi}\pi_i \otimes 1$ for some $\hat{\xi} \in \prod_\alpha$ with $\Pi(d\hat{\xi}) = \xi$. By definition of $\Sigma_{\nu^i}$ as the image of the canonical projection $\prod_{\nu^i} \rightarrow \Sigma_n$, we have $\hat{\xi} \in \pi_i\Sigma_\lambda\pi_i^{-1}$. So there is a $\vartheta \in \Sigma_\lambda$ such that $\hat{\xi}\pi_i = \pi_i\vartheta$. Therefore we have $\psi((\eta' \otimes 1)^{(i)}) = \eta \otimes \pi_i\vartheta \otimes 1 = \eta \otimes \pi_i \otimes 1 = \psi((\eta \otimes 1)^{(i)})$ and $\psi=\varphi^{-1}$ is well-defined.    
\end{proof}

\begin{cor} Layer restriction of a permutation module is isomorphic to a direct sum of permutation modules. In particular, layer restriction of a permutation module has a dual Specht filtration.
\end{cor}
 
\begin{proof}
For $n > l$, we have $res_nInd_lM^\lambda = 0$, so the statement is true. For $n \leq l$, we can apply Lemmas \ref{lemma: arxiv Uv}, \ref{lemma: tensor decomposition of Uv} and Proposition \ref{prop: partition algebras} to get a decomposition $$res_nInd_lM^\lambda \simeq \bigoplus\limits_{v \in V_n^l/_\sim} \bigoplus\limits_{i=1}^{q(v)}( k\Sigma_n \tensorover{k\Sigma_{\nu^i(v)}}k) \in \mathcal{F}_n(S).$$
\end{proof}

This shows that assumption (\ref{assumption: restriction of pm}) is satisfied and we can conclude the following theorem. 

\begin{thm}\label{thm: partition algebras}
Let $r \in \mathbb{N}$ and let $k$ be an algebraically closed field and let $\text{\textnormal{char}}k$ be zero or at least $\max\{5,\lfloor \frac{r}{3} \rfloor\}$. Then the partition algebra $P_k(r,\delta)$, with $\delta \neq 0$, has permutation modules $M(l,\lambda)$, which are a direct sum of indecomposable Young modules. The Young modules are relative projective covers of the cell modules $ind_lS_\lambda$. Every module admitting a cell filtration has well-defined filtration multiplicities. 
\end{thm}

\bibliographystyle{alpha}
\bibliography{referencesFinal}

\begin{thebibliography}{HHKP10}

\bibitem[DJ86]{DJ}
R.~Dipper and G.~D. James.
\newblock Representations of {H}ecke algebras of general linear groups.
\newblock {\em Proc. London Math. Soc. (3)}, 52(1):20--52, 1986.

\bibitem[DR92]{DR}
V.~Dlab and C.~M. Ringel.
\newblock The module theoretical approach to quasi-hereditary algebras.
\newblock In {\em Representations of algebras and related topics ({K}yoto,
  1990)}, volume 168 of {\em London Mathematical Society Lecture Note Series},
  pages 200--224. Cambridge Univ. Press, Cambridge, 1992.

\bibitem[Erd93]{ErdSchurFinalType}
K.~Erdmann.
\newblock Schur algebras of finite type.
\newblock {\em Quart. J. Math. Oxford Ser. (2)}, 44(173):17--41, 1993.

\bibitem[GL96]{GL}
J.~J. Graham and G.~I. Lehrer.
\newblock Cellular algebras.
\newblock {\em Invent. Math.}, 123(1):1--34, 1996.

\bibitem[HHKP10]{HHKP}
R.~Hartmann, A.~Henke, S.~K{\"o}nig, and R.~Paget.
\newblock Cohomological stratification of diagram algebras.
\newblock {\em Math. Ann.}, 347(4):765--804, 2010.

\bibitem[HK12]{HK1}
A.~Henke and S.~K{\"o}nig.
\newblock Schur algebras of {B}rauer algebras {I}.
\newblock {\em Math. Z.}, 272(3-4):729--759, 2012.

\bibitem[HN04]{HN}
D.~J. Hemmer and D.~K. Nakano.
\newblock Specht filtrations for {H}ecke algebras of type {A}.
\newblock {\em J. London Math. Soc. (2)}, 69(3):623--638, 2004.

\bibitem[HP06]{HP}
R.~Hartmann and R.~Paget.
\newblock Young modules and filtration multiplicities for brauer algebras.
\newblock {\em Math. Z.}, 254(2):333--357, 2006.

\bibitem[Jam76]{Jirred}
G.~D. James.
\newblock The irreducible representations of the symmetric groups.
\newblock {\em Bull. London Math. Soc.}, 8(3):229--232, 1976.

\bibitem[Jam78]{JamesLectureNotes}
G.~D. James.
\newblock {\em The representation theory of the symmetric groups}, volume 682
  of {\em Lecture Notes in Mathematics}.
\newblock Springer, Berlin, 1978.

\bibitem[Jam83]{Jtriv}
G.~D. James.
\newblock Trivial source modules for symmetric groups.
\newblock {\em Arch. Math. (Basel)}, 41(4):294--300, 1983.

\bibitem[KX99]{KXinflation}
S.~K{\"o}nig and C.C. Xi.
\newblock Cellular algebras: Inflations and morita equivalences.
\newblock {\em J. London Math. Soc. (2)}, 60(3):700--722, 1999.

\bibitem[Mar93]{Martin}
S.~Martin.
\newblock {\em Schur algebras and representation theory}, volume 112 of {\em
  Cambridge Tracts in Mathematics}.
\newblock Cambridge University Press, Cambridge, 1993.

\bibitem[Mat99]{Mathas}
A.~Mathas.
\newblock {\em Iwahori-{H}ecke algebras and {S}chur algebras of the symmetric
  group}, volume~15 of {\em University Lecture Series}.
\newblock American Mathematical Society, Providence, RI, 1999.

\bibitem[Ngu13]{DungPhD}
D.~T. Nguyen.
\newblock {\em The $q$-Brauer algebras}.
\newblock PhD thesis, Stuttgart, 2013.
\newblock \url{http://dx.doi.org/10.18419/opus-5094}.

\bibitem[Pag07]{P}
R.~Paget.
\newblock A family of modules with {S}pecht and dual {S}pecht filtrations.
\newblock {\em J. Algebra}, 312(2):880--890, 2007.

\bibitem[Pau16]{arxiv}
I.~Paul.
\newblock Restricting cell modules of partition algebras, 2016.
\newblock arXiv:1606.06889.

\bibitem[Rin91]{Ringel}
C.~M. Ringel.
\newblock The category of modules with good filtrations over a quasi-hereditary
  algebra has almost split sequences.
\newblock {\em Math. Z.}, 208(2):209--223, 1991.

\bibitem[Wen12]{Wenzl}
H.~Wenzl.
\newblock A {$q$}-{B}rauer algebra.
\newblock {\em J. Algebra}, 358:102--127, 2012.

\bibitem[Xi99]{Xi}
C.C. Xi.
\newblock Partition algebras are cellular.
\newblock {\em Compositio Math.}, 119(1):99--109, 1999.

\end{thebibliography}

\appendix
\renewcommand{\thesection}{\Roman{section}}

\section{Example for Proposition \ref{prop: partition algebras}, calculated by hand}\label{appendix: example}

\begin{ex} Let $v= \xymini{\circ & \circ & \circ & \circ \tra[r] & \circ & \circ \tra[r] & \circ & \bullet \tra[r] & \bullet} \in V_5^9$. The summand $U_v$ of $e_5(A/J_4)e_9 \tensorover{k\Sigma_{(7,2)}} k$ is isomorphic to $$(k\Sigma_5 \tensorover{k\Sigma_{(3,2)}} k)^2 \oplus (k\Sigma_5 \tensorover{k\Sigma_{(3,1^2)}} k)^2 \oplus (k\Sigma_5 \tensorover{k\Sigma_{(2^2,1)}} k)^2 \oplus (k\Sigma_5 \tensorover{k\Sigma_{(2,1^3)}}k).$$ 

This can be verified as follows. We have $\prod_\alpha \times \prod_\beta = \Sigma_3 \times (\Sigma_2 \times \Sigma_2) \times \Sigma_2$ and the set of double coset representatives is $$\{id, (7\, 8), (6\, 8)(7\, 9), (5\, 8\, 6)(7\, 9), (3\, 8\, 7\, 6\, 5\, 4), (3\, 8\, 6\,5 \, 4)(7\,9), (2\,8\,6\,4)(3\,9\,7\,5)\}.$$ 

The only transpositions in $\Sigma_\lambda=\Sigma_{(7,2)}$ leaving $v\pi$ invariant are those with both end points belonging to the same part $\lambda_i$. In the partial diagram $v\pi$, mark these dots as $\ast$ if they were labelled. The only products of two disjoint transpositions $(a\, b)(c\, d)$ leaving $v\pi$ invariant are those where $a$ and $c$ (or $a$ and $d$) belong to the same part $\lambda_i$ and $b$ and $d$ (or $b$ and $c$, respectively) belong to the same part $\lambda_j$. Note that here, $\lambda_i =\lambda_j$ is possible\footnote{In this case, the transpositions $(a\, b)$ and $(c\, d)$ belong to the first group of dots ($\ast$ or $\star$) as well.}. Mark these dots as $\diamond$ if they were labelled. Put vertical lines at the end of each part $\lambda_i$. Translate this back to $v=v\pi\pi^{-1}$. We can read off $\Pi_\nu = \prod_\alpha \cap \pi \Sigma_\lambda \pi^{-1}$ from the given information by the labelling of dots: $\ast$s  of the same size become symmetric groups, $\diamond$s become wreath products, if both end points $a,b$ lie in the same part $\lambda_i$ in $v\pi$ and the group generated by $(a\, b)(c\, d)$ otherwise\footnote{It does not make a difference for $\Sigma_\nu$ which of the two cases we have, since the projection onto $\Sigma_n$ is the same.}. We do this for each double coset representative in Table \ref{table}.

\begin{table}[b]\caption{Diagrammatic deduction of Young subgroups $\Sigma_\nu$.}\label{table}
\begin{tabular}{| >{$}c<{$} | >{$}c<{$} | >{$}c<{$} | >{$}c<{$} | >{$}c<{$} |}
\hline
 \multirow{2}{*}{$\pi_i$} & v\pi_i & \multirow{2}{*}{$\Pi_\nu$} & \multirow{2}{*}{$\Sigma_\nu$} \\ & v &&\\ \hline\hline 
 \multirow{2}{*}{1} & \xytiny{\ast & \ast & \ast & \diamond \tra[r] & \diamond & \diamond \tra[r] & \diamond \hspace{1ex}\vert \hspace{-1ex} & \bullet \tra[r] & \bullet}  
 & \multirow{2}{*}{$\Sigma_3 \times (\Sigma_2\wr \Sigma_2) = \prod_\alpha$} & \multirow{2}{*}{$\Sigma_{(3,2)}$} \\ 
 & \xytiny{\ast & \ast & \ast & \diamond \tra[r] & \diamond & \diamond \tra[r] & \diamond \hspace{1ex}\vert \hspace{-1ex} & \bullet \tra[r] & \bullet} && \\\hline
 \multirow{2}{*}{$(7\,8)$} & 
\xytiny{\ast & \ast & \ast & \ast \tra[r] & \ast & \circ \arcu[rr] & \bullet \arcd[rr]\hspace{1ex}\vert \hspace{-1ex} & \circ & \bullet} & 
\multirow{2}{*}{$\Sigma_{(3,2,1^2)}$} & \multirow{2}{*}{$\Sigma_{(3,1^2)}$} \\ 
&\xytiny{\ast & \ast & \ast & \ast \tra[r] & \ast & \circ \tra[r] & \circ \hspace{1ex}\vert \hspace{-1ex} & \bullet \tra[r] & \bullet} & & \\\hline
\multirow{2}{*}{$(6\, 8)(7\,9)$} & 
\xytiny{\ast & \ast & \ast & \ast \tra[r] & \ast & \bullet \tra[r] & \bullet \hspace{1ex}\vert \hspace{-1ex} & \ast \tra[r] & \ast} & 
\multirow{2}{*}{$\Sigma_{(3,2^2)}$} & \multirow{2}{*}{$\Sigma_{(3,1^2)}$} \\ 
&\xytiny{\ast & \ast & \ast & \ast \tra[r] & \ast & \ast \tra[r] & \ast \hspace{1ex}\vert \hspace{-1ex} & \bullet \tra[r] & \bullet} & & \\ \hline
\multirow{2}{*}{$(5\, 8\, 6)(7\,9)$} & 
\xytiny{\ast & \ast & \ast & \diamond \arcd[rrrr] & \diamond \arcu[rrrr] & \bullet \tra[r] & \bullet \hspace{1ex}\vert \hspace{-1ex} & \diamond & \diamond} & 
\multirow{2}{*}{$\Sigma_3 \times \langle (4\, 6)(5\, 7)\rangle$} & \multirow{2}{*}{$\Sigma_{(3,2)}$} \\
&\xytiny{\ast & \ast & \ast & \diamond \tra[r] & \diamond & \diamond \tra[r] & \diamond \hspace{1ex}\vert \hspace{-1ex} & \bullet \tra[r] & \bullet} & & \\  \hline
 \multirow{2}{*}{$(3\, 8\, 7\, 6\, 5\, 4)$} & 
\xytiny{\ast & \ast & \diamond \tra[r] & \diamond & \diamond \tra[r] & \diamond  & \bullet \arcu[rr] \hspace{1ex}\vert \hspace{-1ex} & \circ  & \bullet} & 
\multirow{2}{*}{$\Sigma_{(2,1)} \times (\Sigma_2 \wr \Sigma_2)$} & \multirow{2}{*}{$\Sigma_{(2,1,2)}$} \\ 
&\xytiny{\ast & \ast & \circ & \diamond \tra[r] & \diamond & \diamond \tra[r] & \diamond \hspace{1ex}\vert \hspace{-1ex} & \bullet \tra[r] & \bullet} & &\\ \hline
 \multirow{2}{*}{$(3\, 8 \, 6\, 5\, 4)(7\,9)$} & 
\xytiny{\ast & \ast & \ast \tra[r] & \ast & \circ \arcu[rrrr] & \bullet \tra[r] & \bullet \hspace{1ex}\vert \hspace{-1ex} & \circ \tra[r] & \circ} & 
\multirow{2}{*}{$\Sigma_{(2,1,2,1^2)}$} & \multirow{2}{*}{$\Sigma_{(2,1^3)}$} \\ 
&\xytiny{\ast & \ast & \circ & \ast \tra[r] & \ast & \circ \tra[r] & \circ \hspace{1ex}\vert \hspace{-1ex} & \bullet \tra[r] & \bullet} & &\\ \hline
 \multirow{2}{*}{$(2\, 8\, 6\, 4)(3\, 9\, 7\, 5)$} & 
\xytiny{\circ & \diamond \tra[r] & \diamond & \diamond \tra[r] & \diamond & \bullet \tra[r] & \bullet \hspace{1ex}\vert \hspace{-1ex} & \ast & \ast} & 
\multirow{2}{*}{$\Sigma_{(1,2)}\times (\Sigma_2 \wr \Sigma_2)$} & \multirow{2}{*}{$\Sigma_{(1,2^2)}$} \\ 
&\xytiny{\circ & \ast & \ast & \diamond \tra[r] & \diamond & \diamond \tra[r] & \diamond \hspace{1ex}\vert \hspace{-1ex} & \bullet \tra[r] & \bullet} & &\\ \hline
 \end{tabular}
\end{table}
\end{ex}

\section{GAP code to compute summands of restriction of permutation modules for partition algebras}\label{appendix: code}

For a given summand $U_v$ of $e_n(A/J_{n-1})e_l$, the following GAP code calculates which Young subgroups $\Sigma_{\nu^i}$ appear in the decomposition of ${k\Sigma_n \tensorover{k(\prod_\alpha \times \prod_\beta)} k\Sigma_l \tensorover{k\Sigma_\lambda} k} \simeq {U_v \tensorover{k\Sigma_\lambda} k }$, given in Proposition \ref{prop: partition algebras}.  

As input, we need \texttt{G}$=\Sigma_l$, \texttt{H}$=\prod_\alpha \times \prod_\beta$ and \texttt{K}$=\Sigma_\lambda$, as well as the list \texttt{imgs} of images of the generators of $H$ under the canonical epimorphism $\prod_\alpha \times \prod_\beta \twoheadrightarrow \Sigma_n$, sending $\zeta$ to $\Pi(d\zeta)$. We state the code for the example in Appendix \ref{appendix: example}.\\

\noindent INPUT: \texttt{S2:=SymmetricGroup(2); S3:=SymmetricGroup(3); \newline
 S5:=SymmetricGroup(5); S7:=SymmetricGroup(7);} \qquad\qquad \# abbreviations 

\noindent\texttt{G:= SymmetricGroup(14); H:=DirectProduct(S3,WreathProduct(S2,S2),S2); K:=DirectProduct(S7,S2);} \qquad\qquad\qquad\qquad \# $G=\Sigma_l$, $H=\prod_\alpha \times \prod_\beta$, $K=\Sigma_\lambda$.\\

\noindent\texttt{gens:=GeneratorsOfGroup(H); \newline
imgs:=[(1,2,3),(1,2),(),(),(4,5),()];} \quad 
\begin{minipage}{5.5cm}\# to each generator \texttt{gens[i]}, set \newline \# \texttt{imgs[i]}:=the image of \texttt{gens[i]} \# under the canonical epimorphism \# $\prod_\alpha\times\prod_\beta \twoheadrightarrow \Sigma_5$. \end{minipage}\newline
\texttt{hom:=GroupHomomorphismByImages(H,S5,gens,imgs); }\\

\noindent\texttt{iso=function(G,H,K)} \newline
\texttt{local L, r, R, Pinu, Snu;} \newline
\texttt{L:=[]; R:=List(DoubleCosets(G,H,K),Representative);}\newline
\texttt{for r in R do}\newline
\texttt{Pinu$:=$Intersection(H,ConjugateSubgroup(K,r\textasciicircum-1));}\newline
\texttt{Snu$:=$Image(hom,Pinu);}\newline
\texttt{Add(L,Snu);}\newline
\texttt{od;}\newline
\texttt{return L;}\newline
\texttt{end;}\newline

\noindent OUTPUT: list \texttt{L} of all appearing Young subgroups $\Sigma_{\nu^i}=$\texttt{Snu} of $\Sigma_5=$\texttt{S5}.

\end{document}